\documentclass[a4paper, 12pt]{amsart}
\usepackage{amsmath, amssymb, amscd, enumerate, hyperref, MnSymbol, wasysym, tikz, stmaryrd, tensor}
\usepackage[all,cmtip]{xy}

\newtheorem{theorem}{Theorem}[section]

\newtheorem{lemma}[theorem]{Lemma}
\newtheorem{corollary}[theorem]{Corollary}

\theoremstyle{remark}
\newtheorem*{remark}{Remark}

\newtheorem{example}{Example}[section]
\newtheorem*{definition}{Definition}

\theoremstyle{claim}
\newtheorem{claim}{Claim}

\theoremstyle{claim'}
\newtheorem{claim'}{Claim}

\makeatletter
\def\@eqnnum{(\theequation)}
\makeatother

\numberwithin{equation}{section}

\begin{document}

\title[Recurrence relations for traces of singular moduli]
{Recurrence relations satisfied by the traces of singular moduli for $\Gamma_0(N)$}

\author{Bumkyu Cho}
\address{Department of Mathematics, Dongguk University--Seoul, 30 Pildong--ro 1--gil, Jung--gu, Seoul, 04620, South Korea}

\email{bam@dongguk.edu}


\thanks{The author was supported by NRF--2018R1A2B6001645 and the Dongguk University Research Fund of 2019.}


\dedicatory{}

\begin{abstract}
We compute the divisor of the modular equation on the modular curve $\Gamma_0(N) \backslash \mathbb H^*$ and then find recurrence relations satisfied by the modular traces of the Hauptmodul for any congruence subgroup $\Gamma_0(N)$ of genus zero. We also introduce the notions and properties of $\Gamma$-equivalence and $\Gamma$-reduced forms about binary quadratic forms. Using these, we can explicitly compute the recurrence relations for $N = 2, 3, 4, 5$.
\end{abstract}

\maketitle

\section{Introduction}

Let $j(\tau)$ be the Hauptmodul for the full modular group $\mathrm{SL}_2(\mathbb Z)$ whose Fourier expansion is given as
\[ j(\tau) \ = \ \frac{1}{q} \ + \ 196884q \ + \ 21493760 q^2 \ + \ \cdots \qquad (q = e^{2 \pi i \tau}, \ \tau \in \mathbb H). \]
Let $\mathrm{Q}_D$ denote the set of positive definite quadratic forms of discriminant $D$ with the usual action of $\mathrm{SL}_2(\mathbb Z)$. For any $Q \in \mathrm{Q}_D$ we denote by $\tau_Q$ its unique root on $\mathbb H$, and put $\omega_Q = |\overline{\mathrm{SL}_2(\mathbb Z)}_Q|$. We further define the Hurwitz-Kronecker class number $H(D)$ and the trace $t(D)$ of singular moduli as
\[ H(D) \ = \ \sum_{[Q] \in \mathrm{Q}_D / \mathrm{SL}_2(\mathbb Z)} \frac{1}{\omega_Q}, \qquad t(D) \ = \ \sum_{[Q] \in \mathrm{Q}_D / \mathrm{SL}_2(\mathbb Z)} \frac{1}{\omega_Q} j(\tau_Q). \]
In (5) and (7) of \cite[Theorem 2]{Zagier}, Zagier obtained the recurrence relations for $H(D)$ and $t(D)$ by computing the divisor of the modular equation on the modular curve $\mathrm{SL}_2(\mathbb Z) \backslash \mathbb H^*$
\begin{eqnarray*}
\textrm{(1)} & \hspace{-.5cm} & \sum_{|r| < 2 \sqrt{n}} H(r^2 - 4n) \ = \ \sum_{d|n} \mathrm{max}\{d, \, n/d\} \ + \ \left\{ \begin{array}{ll}
1/6 & \mbox{if $n$ is a perfect square} \\
0 & \mbox{otherwise}
\end{array} \right. \\
\textrm{(2)} & \hspace{-.5cm} & \sum_{|r| < 2 \sqrt{n}} t(r^2 - 4n) \ = \ \left\{ \begin{array}{ll}
-4 & \mbox{if $n$ is a perfect square} \\
2 & \mbox{if $4n + 1$ is a perfect square} \\
0 & \mbox{otherwise.}
\end{array} \right.
\end{eqnarray*}

This result was generalized by Choi and Kim \cite{CK} to the case for weakly holomorphic modular functions and weak Maass forms of weight zero on $\Gamma_0^*(p)$. They could find recursions for traces by an analytic way, namely by finding recursions for the Fourier coefficients of weak Jacobi forms of weight $2$ and index $p$ on $\mathrm{SL}_2(\mathbb Z)$. Recently, Murakami \cite{Mura} obtained a similar result for the class number relations from the aspects of moduli spaces of elliptic curves with level structures, and showed that these intersection numbers can be written by the Fourier coefficients of the Siegel Eisenstein series of degree $2$ and weight $2$ with respect to $\mathrm{Sp}_2(\mathbb Z)$.

In the present article, we generalize Zagier's result to the case for the Hauptmodul $j_N(\tau)$ on the congruence subgroup $\Gamma_0(N)$. Our proof will be essentially geometric. We compute the divisor of the modular equation on the modular curve $X_0(N) := \Gamma_0(N) \backslash \mathbb H^*$ and express the modular equation as a product in terms of the values of $j_N(\tau)$ at all cusps in $X_0(N)$ other than $\infty$ and at all CM points in $X_0(N)$.

We denote by $\mathrm{Q}_{D, N}$ the set of positive definite quadratic forms $ax^2 + bxy + cy^2$ of discriminant $D$ such that $a \equiv 0 \pmod{N}$. Then $\Gamma_0(N)$ acts on $\mathrm{Q}_{D, N}$ in a natural way. Assume that $j_N(\tau)$ is of the form
\[ j_N(\tau) \ = \ \frac{1}{q} \ + \ c_1q \ + \ c_2 q^2 \ + \ \cdots \qquad (c_m \in \mathbb Z). \]
Observe that $j_N(\tau)$ has no constant term in its $q$-expansion. The Hurwitz-Kronecker class number $H(D, N)$ and the trace $t(D, N)$ of singular moduli are defined as
\begin{eqnarray*}
H(D, N) & = & \sum_{[Q] \in \mathrm{Q}_{D, N} / \Gamma_0(N)} \frac{1}{\omega_{Q, N}} \\
t(D, N) & = & \sum_{[Q] \in \mathrm{Q}_{D, N} / \Gamma_0(N)} \frac{1}{\omega_{Q, N}} j_N(\tau_Q),
\end{eqnarray*}
where $\omega_{Q, N} = |{\overline{\Gamma}_0(N)}_Q|$.

Zagier himself generalizes his result described above to several ways. For one of them, he deals with $\mathrm{Q}_{D, N, \beta}/\Gamma_0(N)$ to compute the traces of the Hauptmodul $j_N^*(\tau)$ for $\Gamma_0^*(N)$ (see \cite[Section 8]{Zagier}). However, we deal with $\mathrm{Q}_{D, N} / \Gamma_0(N)$ to compute the traces of the Hauptmodul $j_N(\tau)$ for $\Gamma_0(N)$. While the modular curve $X_0^*(N) := \Gamma_0^*(N) \backslash \mathbb H^*$ has only one cusp $\infty$, our modular curve $X_0(N)$ has some cusps other than $\infty$ when $N > 1$. Due to this difference, the values of $j_N(\tau)$ at all cusps other than $\infty$ will appear in our results.

Moreover, we introduce the notions and properties of $\Gamma$-equivalence and $\Gamma$-reduced forms about binary quadratic forms in Section 3 in order to find the representatives for $\mathrm{Q}_{D} / \Gamma$, where $\Gamma$ is a congruence subgroup of $\mathrm{SL}_2(\mathbb Z)$. These notions coincide with the usual proper equivalence and reduced forms when $\Gamma = \mathrm{SL}_2(\mathbb Z)$. These notions combined with the fundamental region for $\Gamma_0(p)$ presented in Section 4 will enable us to systematically compute the representatives for $\mathrm{Q}_{D, p} / \Gamma_0(p)$ for any prime $p$.

\begin{theorem}\label{theorem_main_p}
Let $p$ be a prime number such that the genus of $\Gamma_0(p)$ is zero, and let $n$ be a positive integer relatively prime to $p$. Then we have
\begin{eqnarray*}
\textrm{(1)} & \hspace{-.5cm} & \sum_{|r| < 2 \sqrt{n}} H(r^2 - 4n, \, p) \ = \ 2\sigma(n) \ - \ 2\sum_{d|n} \mathrm{min}\{d, \, n/d\} \ + \ \left\{ \begin{array}{ll}
\sum_{|r| < 2} H(r^2 - 4, \, p) & \mbox{if $n = \Box$} \\
0 & \mbox{otherwise}
\end{array} \right. \\
\textrm{(2)} & \hspace{-.5cm} & \sum_{|r| < 2 \sqrt{n}} t(r^2 - 4n, \, p) \ = \ -j_p(0) \sum_{d|n} \mathrm{min}\{d, \, n/d\} \ + \ \left\{ \begin{array}{ll}
j_p(0) \ + \ \sum_{|r| < 2} t(r^2 - 4, \, p) & \mbox{if $n = \Box$} \\
2 & \mbox{if $4n + 1 = \Box$} \\
0 & \mbox{otherwise.}
\end{array} \right.
\end{eqnarray*}
Here, $\Box$ denotes a perfect square.
\end{theorem}

\begin{remark}
The modular curve $X_0(p)$ has two cusps $\infty$ and $0$. Observe that the value of $j_p(\tau)$ at the cusp $0$ appears in (2).
\end{remark}

Let $N = 2, 3, 4, 5, 7, 9, 13, 25$. Then $j_N(\tau)$ is given as
\[ j_N(\tau) \ = \ \left(\frac{\eta(\tau)}{\eta(N\tau)}\right)^{\frac{24}{N-1}} \ + \ \frac{24}{N-1}, \]
where $\eta(\tau) = q^{\frac{1}{24}} \prod_{n=1}^\infty (1 - q^n)$ is the Dedekind eta function. Using the transformation formula $\eta(-1/\tau) = \sqrt{-i \tau} \eta(\tau)$, we are able to compute the value of $j_N(\tau)$ at the cusp $0$ as $j_N(0) = \frac{24}{N - 1}$.

\begin{example}[$N = 2$]
Let $n \in \mathbb N$ be odd. According to Example \ref{Example - 2, 3, 4} in Section 3, we have $\mathrm{Q}_{-3, 2} = \phi$ and $\mathrm{Q}_{-4, 2}/\Gamma_0(2) = \{ [2x^2 + 2xy + y^2] \}$. Because $(-1 + i)/2$ is an elliptic point of order $2$ for $\Gamma_0(2)$, we have $H(-3, 2) = t(-3, 2) = 0$ and $H(-4, 2) = 1/2$. Moreover, we can compute that $j_2(\frac{-1+i}{2}) = -40$, from which we get $t(-4, 2) = -20$. Hence we obtain
\begin{eqnarray*}
\textrm{(1)} & \hspace{-.5cm} & \sum_{|r| < 2 \sqrt{n}} H(r^2 - 4n, \, 2) \ = \ 2\sigma(n) \ - \ 2\sum_{d|n} \mathrm{min}\{d, \, n/d\} \ + \ \left\{ \begin{array}{ll}
1/2 & \mbox{if $n = \Box$} \\
0 & \mbox{otherwise}
\end{array} \right. \\
\textrm{(2)} & \hspace{-.5cm} & \sum_{|r| < 2 \sqrt{n}} t(r^2 - 4n, \, 2) \ = \ -24 \sum_{d|n} \mathrm{min}\{d, \, n/d\} \ + \ \left\{ \begin{array}{ll}
4 & \mbox{if $n = \Box$} \\
2 & \mbox{if $4n + 1 = \Box$} \\
0 & \mbox{otherwise.}
\end{array} \right.
\end{eqnarray*}
\end{example}

\begin{example}[$N = 3$]
Let $n \in \mathbb N$ be relatively prime to $3$. We obtain that $\mathrm{Q}_{-3, 3}/\Gamma_0(3) = \{ [3x^2 + 3xy + y^2] \}$ and $\mathrm{Q}_{-4, 3} = \phi$, and that $(-3 + \sqrt{-3})/6$ is an elliptic point of order $3$ for $\Gamma_0(3)$. Thus $H(-3, 3) = 1/3$ and $H(-4, 3) = t(-4, 3) = 0$. We further have $t(-3, 3) = -5$ because $j_3(\frac{-3+\sqrt{-3}}{6}) = -15$. Hence we get
\begin{eqnarray*}
\textrm{(1)} & \hspace{-.5cm} & \sum_{|r| < 2 \sqrt{n}} H(r^2 - 4n, \, 3) \ = \ 2\sigma(n) \ - \ 2\sum_{d|n} \mathrm{min}\{d, \, n/d\} \ + \ \left\{ \begin{array}{ll}
2/3 & \mbox{if $n = \Box$} \\
0 & \mbox{otherwise}
\end{array} \right. \\
\textrm{(2)} & \hspace{-.5cm} & \sum_{|r| < 2 \sqrt{n}} t(r^2 - 4n, \, 3) \ = \ -12 \sum_{d|n} \mathrm{min}\{d, \, n/d\} \ + \ \left\{ \begin{array}{ll}
2 & \mbox{if $n = \Box$ or $4n + 1 = \Box$} \\
0 & \mbox{otherwise.}
\end{array} \right.
\end{eqnarray*}
\end{example}

\begin{example}[$N = 5$]
Assume that $n \in \mathbb N$ is relatively prime to $5$. Applying Example \ref{Example - p} in Section 3, we see that $\mathrm{Q}_{-3, 5} = \phi$ and
\[ \mathrm{Q}_{-4, 5}/\Gamma_0(5) \ = \ \{ [5x^2 + 4xy + y^2], \ [5x^2 - 4xy + y^2] \}. \]
If we put $Q = 5x^2 + 4xy + y^2$ and $Q' = 5x^2 - 4xy + y^2$, then both of $\tau_{Q} = (-2 + i)/5$ and $\tau_{Q'} = (2 + i)/5$ are turned out to be elliptic points of order $2$ for $\Gamma_0(5)$. Thus $\omega_{Q, 5} = \omega_{Q', 5} = 2$ and hence $H(-3, 5) = t(-3, 5) = 0$ and $H(-4, 5) = 1$. We also have $t(-4, 5) = -5$ because $j_5(\tau_Q) = -5 + 2i$ and $j_5(\tau_{Q'}) = -5 - 2i$. Therefore Theorem \ref{theorem_main_p} says that
\begin{eqnarray*}
\textrm{(1)} & \hspace{-.5cm} & \sum_{|r| < 2 \sqrt{n}} H(r^2 - 4n, \, 5) \ = \ 2\sigma(n) \ - \ 2\sum_{d|n} \mathrm{min}\{d, \, n/d\} \ + \ \left\{ \begin{array}{ll}
1 & \mbox{if $n = \Box$} \\
0 & \mbox{otherwise}
\end{array} \right. \\
\textrm{(2)} & \hspace{-.5cm} & \sum_{|r| < 2 \sqrt{n}} t(r^2 - 4n, \, 5) \ = \ -6 \sum_{d|n} \mathrm{min}\{d, \, n/d\} \ + \ \left\{ \begin{array}{ll}
1 & \mbox{if $n = \Box$} \\
2 & \mbox{if $4n + 1 = \Box$} \\
0 & \mbox{otherwise.}
\end{array} \right.
\end{eqnarray*}
\end{example}

We also obtain results similar to Theorem \ref{theorem_main_p}, in which the level need not be a prime number. See Theorem \ref{theorem_main1}, Corollary \ref{corollary_main1}, and Theorem \ref{theorem_main2} for detailed statements.

\begin{example}[$N = 4$]
Assume that $n \in \mathbb N$ is odd. The $\Gamma_0(4)$ has three cusps $\infty$, $0$, and $1/2$, and we can evaluate the values at the cusps as $j_4(0) = 8$ and $j_4(1/2) = -8$. Because $\mathrm{Q}_{-4, 4} = \mathrm{Q}_{-3, 4} = \phi$, we have $H(-4, 4) = H(-3, 4) = t(-4, 4) = t(-3, 4) = 0$ and hence we get
\begin{eqnarray*}
\textrm{(1)} & \hspace{-.5cm} & \sum_{|r| < 2 \sqrt{n}} H(r^2 - 4n, 4) \ = \ 2\sigma(n) \ - \ 3 \sum_{d|n} \mathrm{min}\{d, \, n/d \} \ + \ \left\{ \begin{array}{ll}
1 & \mbox{if $n = \Box$} \\
0 & \mbox{otherwise}
\end{array} \right. \\
\textrm{(2)} & \hspace{-.5cm} & \sum_{|r| < 2 \sqrt{n}} t(r^2 - 4n, 4) \ = \ \left\{ \begin{array}{ll}
2 & \mbox{if $4n + 1 = \Box$} \\
0 & \mbox{otherwise.}
\end{array} \right.
\end{eqnarray*}
\end{example}

\section{Modular equations and traces of singular moduli}

Assume that $\Gamma_0(N)$ is of genus zero. For a positive integer $n$, we define
\begin{eqnarray*}
\mathrm{M}_{n, N} & = & \{ \left(\begin{smallmatrix} a & b \\ c & d \end{smallmatrix} \right) \in \mathrm{M}_2(\mathbb Z) \, | \, ad - bc = n, \ (a, N) = 1, \ c \equiv 0 \pmod N \} \\
j_N(\tau) & = & \frac{1}{q} \ + \ \sum_{m = 1}^\infty c_m q^m \quad (c_m \in \mathbb Z) \quad \mbox{ a Hauptmodul for } \Gamma_0(N) \\
\Phi_n^{\scalebox{0.6}{$\Gamma_0(N)$}}(X, j_N) & = & \prod_{[\alpha] \in \Gamma_0(N) \backslash \mathrm{M}_{n, N}} \big(X \, - \, j_N(\alpha(\tau))\big).
\end{eqnarray*}
Clearly we have $\Phi_1^{\scalebox{0.6}{$\Gamma_0(N)$}}(X, j_N) = X - j_N$, so $\Phi_1^{\scalebox{0.6}{$\Gamma_0(N)$}}(X, Y) = X - Y$. Since all coefficients of the polynomial $\Phi_n^{\scalebox{0.6}{$\Gamma_0(N)$}}(X, j_N)$ in the variable $X$ are the elementary symmetric functions of the $j_N \circ \alpha$'s, they are invariant under $\Gamma_0(N)$. We can further deduce that they are contained in $\mathbb C(j_N)$ because $j_N$ is a Hauptmodul for the congruence subgroup $\Gamma_0(N)$ of genus zero. Thus we may think of $\Phi_n^{\scalebox{0.6}{$\Gamma_0(N)$}}(X, j_N)$ as a polynomial in $\mathbb C(j_N)[X]$.

According to \cite[Proposition 3.36]{Shi}, we can choose the representatives of all the distinct orbits in $\Gamma_0(N) \backslash \mathrm{M}_{n, N}$ as
\[
\alpha_{a, b} \ := \ \bigg( \, \begin{matrix}
a & b \\ 0 & n/a
\end{matrix} \bigg),
\]
where $a | n$, $(a, N) = 1$, and $0 \leq b < n/a$. The number of these representatives is $\sum_{d | n, \, (n/d, N) = 1} d$.

\begin{lemma}\label{lemma_Phi}
With the notation and assumptions as above, we have the following.

(1) $\Phi_n^{\scalebox{0.6}{$\Gamma_0(N)$}}(X, Y) \in \mathbb Q(Y)[X]$ is a polynomial in $X$ over $\mathbb Q(Y)$ of degree $\sum_{d | n, \, (n/d, N) = 1} d$.

(2) $\Phi_n^{\scalebox{0.6}{$\Gamma_0(N)$}}(X, Y) \in \mathbb Z[X, Y]$ if $(n, N) = 1$.

(3) If $(n, N) = 1$ and if $n$ is not a perfect square, then $\Phi_n^{\scalebox{0.6}{$\Gamma_0(N)$}}(X, X)$ has degree $\sum_{d | n} \mathrm{max}\{d, n/d \}$ and leading coefficient $(-1)^{\frac{1}{2} \sigma_0(n)}$.

(4) If $(n, N) = 1$ and if $n$ is a perfect square, then the polynomial $\left.\frac{\Phi_n^{\scalebox{0.6}{$\Gamma_0(N)$}}(X, Y)}{\Phi_1^{\scalebox{0.6}{$\Gamma_0(N)$}}(X, Y)}\right|_{Y = X}$ in $X$ has degree $\sum_{d | n} \mathrm{max}\{d, n/d \} - 1$ and leading coefficient $(-1)^{\frac{1}{2} (\sigma_0(n)-1)} \sqrt{n}$.
\end{lemma}

\begin{proof}
The proof is essentially the same as the proof of \cite[\S2 in Chapter 5]{Lang} or \cite[Theorem 11.18]{Cox}, in which they presented the classical result about the primitive modular equation for $j(\tau)$.

(1) For any integer $k$ relatively prime to $n$, let $\psi_k \in \mathrm{Gal}(\mathbb Q(\zeta_n) / \mathbb Q)$ be an automorphism defined by $\psi_k(\zeta_n) = \zeta_n^k$. Then $\psi_k$ induces an automorphism of $\mathbb Q(\zeta_n)((q^{\frac{1}{n}}))$ in a natural way and we denote it by the same letter $\psi_k$. Because
\[ j_N(\alpha_{a, b}(\tau)) \ = \ \zeta_n^{-ab} q^{-\frac{a^2}{n}} \ + \ \sum_{m = 1}^\infty c_m \zeta_n^{abm} q^{\frac{ma^2}{n}}, \]
we get
\begin{eqnarray*}
\psi_k(j_N(\alpha_{a, b}(\tau))) & = & \zeta_n^{-abk} q^{-\frac{a^2}{n}} \ + \ \sum_{m = 1}^\infty c_m \zeta_n^{abkm} q^{\frac{ma^2}{n}} \\
& = & j_N(\alpha_{a, b'}(\tau)),
\end{eqnarray*}
where $b'$ is the unique integer satisfying $0 \leq b' < n/a$ and $bk \equiv b' \pmod{n/a}$. This shows that all the elementary symmetric functions of the $j_N \circ \alpha_{a, b}$'s are actually contained in $\mathbb Q((q^{\frac{1}{n}}))$. Because $\mathbb C(j_N) \cap \mathbb Q((q^{\frac{1}{n}})) = \mathbb Q(j_N)$, we have $\Phi_n^{\scalebox{0.6}{$\Gamma_0(N)$}}(X, j_N) \in \mathbb Q(j_N)[X]$.

(2) Let $s \in \mathbb Q \cup \{ \infty \}$. Note that
\begin{eqnarray*}
& & j_N \circ \alpha_{a, b} \mbox{ has a pole at the cusp $s$} \\
& \Longleftrightarrow & \alpha_{a, b}(s) \sim \infty \mbox{ under } \Gamma_0(N) \\
& \Longleftrightarrow & s = \frac{\frac{n}{a}i - bj}{aj} \mbox{ for some $i, j \in \mathbb Z$ with $(i, j) = 1$, $(i, N) = 1$, $j \equiv 0 \pmod N$}.
\end{eqnarray*}
Because $(\frac{n}{a}i - bj, N) = 1$ and $aj \equiv 0 \pmod N$, we see that all the elementary symmetric functions of the $j_N \circ \alpha_{a, b}$'s are holomorphic at all cusps of $\Gamma_0(N)$ except $\infty$. They are also holomorphic on $\mathbb H$ clearly and hence we deduce that $\Phi_n^{\scalebox{0.6}{$\Gamma_0(N)$}}(X, j_N)$ are contained in $\mathbb Q[j_N][X]$. Since $j_N \circ \alpha_{a, b}$ has its Fourier coefficients in $\mathbb Z[\zeta_n]$, all Fourier coefficients of the elementary symmetric functions of the $j_N \circ \alpha_{a, b}$'s must be algebraic integers. This shows that $\Phi_n^{\scalebox{0.6}{$\Gamma_0(N)$}}(X, j_N)$ are actually contained in $\mathbb Z[j_N][X]$.

(3) Because $\Phi_n^{\scalebox{0.6}{$\Gamma_0(N)$}}(X, X) \in \mathbb Z[X]$, we have
\[ \Phi_n^{\scalebox{0.6}{$\Gamma_0(N)$}}(j_N, j_N) \, = \, a_\ell q^{-\ell} \, + \, \mbox{(higher degree terms)}, \]
where $\ell = \deg_X \Phi_n^{\scalebox{0.6}{$\Gamma_0(N)$}}(X, X)$ and $a_\ell \in \mathbb Z$. Since
\[ j_N(\tau) \ - \ j_N(\alpha_{a, b}(\tau)) \ = \ q^{-1} \ - \ \zeta_n^{-ab} q^{-\frac{a^2}{n}} \ + \ \mbox{(terms of degree $> 0$)}, \]
we compute
\[ \ell \ = \ \sum_{a|n} \sum_{0 \leq b < n/a} \mathrm{max}\{1, a^2/n \} \ = \ \sum_{a|n} \frac{n}{a} \mathrm{max}\{ 1, a^2/n \} \ = \ \sum_{a|n} \mathrm{max}\{ n/a, a \} \]
and
\[ a_\ell \ = \ \prod_{a|n \atop a > \sqrt{n}} \prod_{0 \leq b < n/a} (-\zeta_n^{-ab}) \ = \ \prod_{a|n \atop a > \sqrt{n}} (-1) \ = \ (-1)^{\frac{1}{2} \sigma_0(n)}. \]

(4) The degree $\ell$ and leading coefficient $a_\ell$ of $\left.\frac{\Phi_n^{\scalebox{0.6}{$\Gamma_0(N)$}}(X, Y)}{\Phi_1^{\scalebox{0.6}{$\Gamma_0(N)$}}(X, Y)}\right|_{Y = X}$ are similarly computed as
\[ \ell \ = \ \sum_{a|n} \sum_{0 \leq b < n/a} \mathrm{max}\{1, a^2/n \} \ - \ 1 \ = \ \sum_{a|n} \mathrm{max}\{ n/a, a \} \ - \ 1 \]
and
\[ a_\ell \ = \ \prod_{a|n \atop a > \sqrt{n}} \prod_{0 \leq b < n/a} (-\zeta_n^{-ab}) \ \times \ \prod_{0 < b < \sqrt{n}} (1 - \zeta_{\small\sqrt{n}}^{-b}) \ = \ (-1)^{\frac{1}{2} (\sigma_0(n)-1)} \sqrt{n}, \]
since the case that $a = \sqrt{n}$ and $b = 0$ needs to be removed in the computation of (3).
\end{proof}

\begin{remark}
The author has also obtained some properties about the primitive modular equations for $\Gamma_1(m) \bigcap \Gamma_0(mN)$ (see \cite[Theorem 2.1]{Cho2}), and more generally for $\Gamma_H(N, t)$ and $\Gamma$ with $[\Gamma: \Gamma_H(N, t)] = 2$ (see \cite[Theorems 2.1 and 2.2]{Cho3}).
\end{remark}

Given a negative integer $D$, we define
\begin{eqnarray*}
\mathrm{Q}_{D, N} & = & \{ ax^2 + bxy + cy^2 \, | \, a > 0, \ b^2 - 4ac = D, \ a \equiv 0 \! \! \! \pmod{N} \} \\
H_{D, N}(X) & = & \prod_{[Q] \in \mathrm{Q}_{D, N} / \Gamma_0(N)} \big( X - j_N(\tau_Q) \big)^{1 / \omega_{Q, N}} \qquad (\omega_{Q, N} \ = \ |\overline{\Gamma}_0(N)_Q|).
\end{eqnarray*}
Here $\mathrm{Q}_{D, N} / \Gamma_0(N)$ denotes the set of all equivalence classes in $\mathrm{Q}_{D, N}$ under the so-called $\Gamma_0(N)$-equivalence defined by
\[ Q \sim Q' \quad \mbox{if} \ Q' = Q \cdot \gamma := Q(ax + by, cx + dy) \mbox{ for some } \gamma = (\begin{smallmatrix} a & b \\ c & d \end{smallmatrix}) \in \Gamma_0(N), \]
and $\tau_Q$ denotes the unique solution of $Q(x, 1) = 0$ on $\mathbb H$ for a quadratic form $Q \in \mathrm{Q}_{D, N}$.

\begin{remark}
The product in the definition of $H_{D, N}(X)$ is a well-defined finite product according to Theorem \ref{Theorem - reduced} and $\tau_{Q \cdot \gamma} = \gamma^{-1}(\tau_Q)$ (see Section 3). It is also worth noticing that $Q = Q'$ if and only if $\tau_Q = \tau_{Q'}$ and both $Q$ and $Q'$ have the same discriminant.
\end{remark}

Let $S$ denote the set of all inequivalent cusps of $\Gamma_0(N)$.  For any cusp $s = \frac{a}{b} \in S$ with $(a, b) = 1$ and $b \geq 0$, we write $b_s := b$ and define
\[ \nu_{s, n, N} \ = \ \sum \min \{d, \, n/d\}, \]
where the summation is over all the positive divisors $d$ of $n$ such that $n \equiv d^2 \pmod {(b_s, \frac{N}{(b_s, N)})}$. (Here we understand that $\frac{\pm 1}{0} = \infty$.)

\begin{theorem}\label{theorem_main1}
Assume that the genus of $\Gamma_0(N)$ is zero and that $n$ is a positive integer relatively prime to $N$. Assume further that $n$ is not a perfect square. Then we have
\[
\Phi_n^{\scalebox{0.6}{$\Gamma_0(N)$}}(X, \, X) \ = \ (-1)^{\frac{1}{2}\sigma_0(n)} \ \times \ \prod_{r \in \mathbb Z \atop |r| < 2 \sqrt{n}} H_{r^2 - 4n, N}(X) \ \times \ \prod_{s \in S \atop s \not \sim \infty} \big(X \, - \, j_N(s)\big)^{\nu_{s, n, N}},
\]
where $S$ is the set of all inequivalent cusps of $\Gamma_0(N)$.
\end{theorem}

To prove this, we need more lemmas. Let $P_0 \in X$ be a point on a compact Riemann surface $X$ of genus zero, and let $t(z)$ be a local parameter at $P_0$, and let $\mathrm{ord}_{P_0}(f)$ denote the order of a meromorphic function $f(z)$ on $X$ at $P_0$. Then we have
\[
f(z) \ = \ \sum_{m = \mathrm{ord}_{P_0}(f)}^\infty c_m t(z)^m
\]
with $c_{\mathrm{ord}_{P_0}(f)} \neq 0$. Suppose that the function field of $X$ is generated by $f$ over $\mathbb C$. Then $f$ has a unique simple pole and a unique simple zero, and all the other points have order $0$. (For a reference, see \cite[Proposition 2.11 (3)]{Shi}.) If $P_0$ is not a simple pole of $f$, then $f(z) - f(P_0)$ must have a simple zero at $P_0$ because $\mathbb C(f(z) - f(P_0))$ is also the function field of $X$. Thus we can deduce the following.

\begin{lemma}\label{lemma_c_1}
Suppose that the function field of a compact Riemann surface $X$ of genus zero is generated by $f$ over $\mathbb C$. Let $t(z)$ be a local parameter at a point $P_0 \in X$. If $P_0$ is not a simple pole of $f$, then we have
\[
f(z) \ = \ f(P_0) \ + \ \sum_{m = 1}^\infty c_m t(z)^m
\]
with $c_1 \neq 0$.
\end{lemma}

Let $\mathrm{M}_{n, N}^{\mathrm{ell}}$ be the subset of $\mathrm{M}_{n, N}$ consisting of all elliptic elements, i.e.
\begin{eqnarray*}
\mathrm{M}_{n, N}^{\mathrm{ell}} & = & \{ \alpha \in \mathrm{M}_{n, N} \, | \ |\mathrm{tr}(\alpha)| < 2 \sqrt{n} \}.
\end{eqnarray*}

\begin{lemma}\label{lemma_j_N}
Let $\tau_0 \in \mathbb H$ be a fixed point of $\alpha \in \mathrm{M}_{n, N}^{\mathrm{ell}} - \sqrt{n} \, \Gamma_0(N)$. Then we have
\[ \lim_{\tau \rightarrow \tau_0} \frac{j_N(\tau) \ - \ j_N(\alpha(\tau))}{j_N(\tau) \ - \ j_N(\tau_0)} \ \neq \ 0, \ \infty. \]
\end{lemma}

\begin{proof}
Set $e = |\overline{\Gamma}_0(N)_{\tau_0}|$. Then $t = (\tau - \tau_0)^e$ is a local parameter at $P_0 := \Gamma_0(N) \tau_0 \in X_0(N)$. According to Lemma \ref{lemma_c_1} we have
\[ j_N(\tau) \ = \ j_N(\tau_0) \ + \ \sum_{m = 1}^\infty c_m (\tau - \tau_0)^{em} \]
with $c_1 \neq 0$. We see that $|\overline{\Gamma'}_{\tau_0}| = |\overline{\Gamma}_0(N)_{\tau_0}| = e$ and hence derive that $t = (\tau - \tau_0)^e$ is also a local parameter at $Q_0 := \Gamma' \tau_0 \in X'$, where $X' = \Gamma' \backslash \mathbb H^*$ and $\Gamma' = \alpha^{-1} \Gamma_0(N) \alpha$. Because $j_N \circ \alpha$ generates the function field $\mathbb C(X')$ of the compact Riemann surface $X'$ over $\mathbb C$, we infer from Lemma \ref{lemma_c_1} again and from $\alpha(\tau_0) = \tau_0$ that
\[ j_N(\alpha(\tau)) \ = \ j_N(\alpha(\tau_0)) \ + \ \sum_{m = 1}^\infty d_m (\tau - \tau_0)^{em} \ = \ j_N(\tau_0) \ + \ \sum_{m = 1}^\infty d_m (\tau - \tau_0)^{em} \]
with $d_1 \neq 0$. Because
\begin{eqnarray*}
\lim_{\tau \rightarrow \tau_0} \frac{j_N(\tau) \ - \ j_N(\alpha(\tau))}{j_N(\tau) \ - \ j_N(\tau_0)} & = & \lim_{\tau \rightarrow \tau_0} \frac{\sum_{m = 1}^\infty (c_m \ - \ d_m) (\tau \ - \ \tau_0)^{em}}{\sum_{m = 1}^\infty c_m (\tau \ - \ \tau_0)^{em}} \\
& = & \frac{c_1 - d_1}{c_1},
\end{eqnarray*}
it is enough to show that $c_1 \neq d_1$.

Substituting $\alpha(\tau)$ for $\tau$ in the first equation of this proof, we get
\begin{eqnarray*}
j_N(\alpha(\tau)) & = & j_N(\tau_0) \ + \ \sum_{m = 1}^\infty c_m (\alpha(\tau) \ - \ \tau_0)^{em} \\
& = & j_N(\tau_0) \ + \ c_1 \left( \frac{\alpha(\tau) \ - \ \alpha(\tau_0)}{\tau \ - \ \tau_0} \right)^e (\tau \ - \ \tau_0)^e \ + \ \sum_{m = 2}^\infty c_m (\alpha(\tau) \ - \ \tau_0)^{em},
\end{eqnarray*}
from which we deduce that
\[ d_1 \ = \ \lim_{\tau \rightarrow \tau_0} \frac{j_N(\alpha(\tau)) \ - \ j_N(\tau_0)}{(\tau \ - \ \tau_0)^e} \ = \ c_1 \alpha'(\tau_0)^e. \]

Now we will show that $\alpha'(\tau_0)^e \neq 1$. Write $\alpha = (\begin{smallmatrix} a & b \\ c & d \end{smallmatrix})$. Since $\alpha'(\tau_0) = n/(c\tau_0 + d)^2$, we infer that
\[ \alpha'(\tau_0)^e = 1 \ \iff \ c\tau_0 + d \ = \ \sqrt{n} \zeta_{2e}^k \qquad \mbox{for some } k = 0, 1, \ldots, 2e-1. \]
Clearly we have $e = 1$, $2$, or $3$. If $e = 1$, then $\sqrt{n} \zeta_{2e}^k$ is real, but $c\tau_0 + d$ is not because $c \neq 0$. So we have $\alpha'(\tau_0) \neq 1$.

From now on, we suppose that $e > 1$. We deal with two cases. Firstly, we consider the case that $n$ is not a perfect square. We assume that $\alpha'(\tau_0)^2 = 1$. Then $c\tau_0 + d = \sqrt{n} i^k$ for some $k \in \mathbb Z$. If $k = 0$ or $2$, then we have a contradiction as above. If $k = 1$ or $3$, then we have $\tau_0 = (-d \pm \sqrt{n} i)/c$, whose imaginary part is irrational because $n$ is not a perfect square. Since every elliptic point of $\Gamma_0(N)$ of order $2$ must be equivalent to $i$ under $\mathrm{SL}_2(\mathbb Z)$, we obtain that $\tau_0 = \gamma (i)$ for some $\gamma \in \mathrm{SL}_2(\mathbb Z)$, from which we infer that the imaginary part of $\tau_0$ must be rational. Thus we have $\alpha'(\tau_0)^2 \neq 1$. Now we assume that $\alpha'(\tau_0)^3 = 1$. Since the real part of $\sqrt{n} \zeta_{6}^k$ is irrational, so is the real part of $\tau_0$. However, we have $\tau_0 = \gamma(\frac{-1 + \sqrt{3}i}{2})$ for some $\gamma \in \mathrm{SL}_2(\mathbb Z)$, whose real part is easily computed to be rational.

Finally, we consider the case that $n$ is a perfect square. Assume that $\alpha'(\tau_0)^e = 1$, i.e. that $c\tau_0 + d = \sqrt{n} \zeta_{2e}^k$ for some $k = 0, 1, \ldots, 2e-1$. Without loss of generality, we may assume that $c > 0$. From $\alpha(\tau_0) = \tau_0$, we have $\tau_0 = \frac{a - d + \sqrt{(a+d)^2 - 4n}}{2c}$. On the other hand, since $|\overline{\Gamma}_0(N)_{\tau_0}| = e > 1$, there is an elliptic element $\gamma = (\begin{smallmatrix} A & B \\ C & D \end{smallmatrix}) \in \Gamma_0(N)$ such that $C > 0$ and $\gamma(\tau_0) = \tau_0$, so we have $\tau_0 = \frac{A-D + \sqrt{(A+D)^2 - 4}}{2C}$. Comparing these with $\tau_0 = \frac{-d + \sqrt{n}\zeta_{2e}^k}{c}$, we derive that $k = 1$, $a + d = 0$, $A + D = 0$, $c = \sqrt{n} C$, $d = \sqrt{n} D$, and $a = \sqrt{n} A$ if $e = 2$ and that $k = 1, 2$, $a + d = (-1)^{k-1} \sqrt{n}$, $A + D = (-1)^{k-1}$, $c = \sqrt{n} C$, $d = \sqrt{n} D$, and $a = \sqrt{n} A$ if $e = 3$. In either case, combining with $ad - bc = n$ and $AD - BC = 1$, we deduce that $b = \sqrt{n}B$ and hence $\alpha = \sqrt{n} \gamma \in \sqrt{n} \, \Gamma_0(N)$. This leads us to a contradiction.
\end{proof}

The unique fixed point of $\alpha \in \mathrm{M}_{n, N}^{\mathrm{ell}}$ on $\mathbb H$ is denoted by $\tau_\alpha$. We also use the main involution of $\mathrm{M}_2(\mathbb Z)$ defined by $\alpha^{\iota} = (\begin{smallmatrix} d & -b \\ -c & a \end{smallmatrix})$ for $\alpha = (\begin{smallmatrix} a & b \\ c & d \end{smallmatrix})$.

\begin{lemma}\label{lemma_q}
Assume that $(n, N) = 1$. For any matrix $\alpha = (\begin{smallmatrix} a & b \\ c & d \end{smallmatrix})$ we denote by $q(\alpha)$ the quadratic form $\mathrm{sgn}(c) \cdot (cx^2 + (d - a)xy - by^2)$. Then $q:\mathrm{M}_{n, N}^{\mathrm{ell}} \longrightarrow \bigcup_{|r| < 2\sqrt{n}} \mathrm{Q}_{r^2 - 4n, N}$ is a surjective map such that

(1) $q\big( \big( \begin{smallmatrix} \frac{-b + r}{2} & -c \\ a & \frac{b + r}{2} \end{smallmatrix} \big) \big) = ax^2 + bxy + cy^2$

(2) $q(\alpha) = q(\alpha^{\iota})$

(3) $\tau_\alpha = \tau_{q(\alpha)}$

(4) $\tau_{\gamma^{-1} \alpha \gamma} = \gamma^{-1}(\tau_\alpha) = \tau_{q(\alpha) \cdot \gamma}$

(5) $q(\gamma^{-1} \alpha \gamma) = q(\alpha) \cdot \gamma$

(6) $q(\alpha) = q(\beta)$ $\iff$ $\beta = \pm \alpha$ or $\beta = \pm \alpha^{\iota}$

(7) $\alpha \neq \alpha^{\iota}$

(8) $\alpha = - \alpha^{\iota}$ $\iff$ $\mathrm{tr}(\alpha) = 0$\\
for any $ax^2 + bxy + cy^2 \in \mathrm{Q}_{r^2 - 4n, N}$, $\alpha, \beta \in \mathrm{M}_{n, N}^{\mathrm{ell}}$, and $\gamma \in \Gamma_0(N)$.
\end{lemma}

\begin{proof}
It is nothing but a tedious computation to verify all of the properties above, so we just remark that the assumption that $(n, N) = 1$ is used only where to verify that $\big( \begin{smallmatrix} (-b + r)/2 & -c \\ a & (b + r)/2 \end{smallmatrix} \big)$ and $\alpha^\iota$ are contained in $\mathrm{dom}(q) = \mathrm{M}_{n, N}^{\mathrm{ell}}$.
\end{proof}

We fix a fundamental region $R$ for $\Gamma_0(N)$ such that the map of $R$ to $Y_0(N) := \Gamma_0(N) \backslash \mathbb H$ given by $\tau \mapsto \Gamma_0(N) \tau$ is a bijection. Let $\overline{\mathrm{M}}_{n, N} = \mathrm{M}_{n, N}/ \{ \pm 1 \}$, $\overline{\Gamma}_0(N) = \Gamma_0(N)/ \{ \pm 1 \}$, $\overline{\mathrm{M}}_{n, N}^{\mathrm{ell}} = \mathrm{M}_{n, N}^{\mathrm{ell}}/ \{ \pm 1 \}$, and $\overline{\mathrm{M}}_{n, N}^{\mathrm{ell}}(R) = \{ \bar{\alpha} \in \overline{\mathrm{M}}_{n, N}^{\mathrm{ell}} \, | \, \tau_\alpha \in R \}$.

\begin{lemma}\label{lemma_qbar}
Assume that $(n, N) = 1$. Then the map
\[ \bar{q}: \overline{\mathrm{M}}_{n, N}^{\mathrm{ell}}(R) \ \longrightarrow \ \{ (r, \, [Q]) \, | \, r \in \mathbb Z, \ |r| < 2\sqrt{n}, \ [Q] \in \mathrm{Q}_{r^2 - 4n, N} / \Gamma_0(N) \} \]
given by $\bar{\alpha} \longmapsto (\mathrm{sgn}(c_\alpha) \cdot \mathrm{tr}(\alpha), \, [q(\alpha)])$ is a bijection. Here, $c_\alpha$ denotes the $(2, 1)$-entry of $\alpha$.
\end{lemma}

\begin{proof}
First of all, $\bar{q}$ is well defined clearly and surjective by Lemma \ref{lemma_q} (1). Suppose that $\bar{q}(\bar{\alpha}) = \bar{q}(\bar{\beta})$. Then we have
\begin{eqnarray*}
q(\alpha) \sim q(\beta) & \Longrightarrow & q(\alpha) = q(\beta) \cdot \gamma \qquad \mbox{for some } \gamma \in \Gamma_0(N) \\
& \Longrightarrow & q(\alpha) = q(\gamma^{-1} \beta \gamma) \qquad \mbox{for some } \gamma \in \Gamma_0(N) \\
& \Longrightarrow & \gamma^{-1} \beta \gamma = \pm \alpha, \ \pm \alpha^\iota \qquad \mbox{for some } \gamma \in \Gamma_0(N)
\end{eqnarray*}
by (5) and (6) of Lemma \ref{lemma_q}. Since $(\gamma^{-1} \beta \gamma)(\tau_\alpha) = \tau_\alpha$, we see that $\gamma(\tau_\alpha) = \tau_\beta$. However $\tau_\alpha, \tau_\beta \in R$ and $\gamma(\tau_\alpha) = \tau_\beta$ imply that $\tau_\beta = \tau_\alpha$. Since any two elliptic elements having the same fixed point on $\mathbb H$ commute each other (e.g. see \cite[Proposition 1.16]{Shi}), we have $\gamma^{-1} \beta \gamma = \beta$ and hence $\bar{\beta} = \bar{\alpha}$, $\overline{\alpha^\iota}$. Note that
\[ \bar{\alpha} = \overline{\alpha^\iota} \ \Longleftrightarrow \ \alpha = - \alpha^\iota \ \Longleftrightarrow \ \mathrm{tr}(\alpha) = 0 \]
by (7) and (8) of Lemma \ref{lemma_q} and that
\[ \mathrm{sgn}(c_\beta) \cdot \mathrm{tr}(\beta) \ = \ \mathrm{sgn}(c_\alpha) \cdot \mathrm{tr}(\alpha) \ = \ -\mathrm{sgn}(c_{\alpha^\iota}) \cdot \mathrm{tr}(\alpha^\iota). \]
Therefore we have $\bar{\beta} = \bar{\alpha}$.
\end{proof}

Instead of dealing with an arbitrary set $S$ in Theorem \ref{theorem_main1}, it sufficies to deal with a specific set $S'$ of all inequivalent cusps of $\Gamma_0(N)$. According to \cite[Proof of Proposition 1.43]{Shi} or \cite[Corollary 4 (1)]{CKP1}, $S'$ may be chosen as a set of all nonnegative rational numbers $u/v \in \mathbb Q$ satisfying the properties that

(1) $v | N$, $0 \leq u < N$, $(u, N) = 1$

(2) $u = u' \, $ if $\, u/v, \, u'/v \in S' \, $ and $\, u \equiv u' \pmod {(v, N/v)}$.

\begin{lemma}\label{lemma_cusp}
Assume that $(n, N) = 1$. For any cusp $u/v \in S'$ we have
\[ \alpha_{a, b}(u/v) \ \sim \ u/v \quad \mbox{under } \Gamma_0(N) \ \iff \ n \ \equiv \ (au + bv, n/a)^2 \pmod{(v, N/v)}. \]
\end{lemma}

\begin{proof}
Let $d = (au + bv, n/a)$. Then we compute
\[ \alpha_{a, b}(u/v) \ = \ \frac{(au + bv)/d}{nv/ad}, \qquad \big( (au + bv)/d, \ nv/ad \big) \ = \ 1. \]
Appealing to \cite[Proof of Proposition 1.43]{Shi} or \cite[Corollary 4 (1)]{CKP1} again, we have
\begin{eqnarray*}
& & \alpha_{a, b}(u/v) \ \sim \ u/v \quad \mbox{under } \Gamma_0(N) \\
& \iff & \exists \, t \in (\mathbb Z / N\mathbb Z)^\times, \ x \in \mathbb Z \ \mbox{ such that } \\
& & \qquad (au + bv)/d \ \equiv \ t^{-1}u + xv \pmod N, \quad nv/ad \ \equiv \ tv \pmod N \\
& \iff & \exists \, t \in (\mathbb Z / N\mathbb Z)^\times, \ x \in \mathbb Z \ \mbox{ such that } \\
& & \qquad (au + bv)t/d \ \equiv \ u + nvx/ad \pmod N, \quad t \ \equiv \ n/ad \pmod {N/v} \\
& \iff & \exists \, t \in \mathbb Z / N\mathbb Z, \ x \in \mathbb Z \ \mbox{ such that } \\
& & \qquad (au + bv)t/d \ \equiv \ u + nvx/ad \pmod N, \quad t \ \equiv \ n/ad \pmod {N/v},
\end{eqnarray*}
because the first congruence equation of the last equivalent statement implies that $(t, v) = 1$ and the second congruence equation implies that $(t, N/v) = 1$. Thus we obtain that
\begin{eqnarray*}
& & \alpha_{a, b}(u/v) \ \sim \ u/v \quad \mbox{under } \Gamma_0(N) \\
& \iff & \exists \, x, y, z \in \mathbb Z \ \mbox{ such that } \, (au + bv)/d \cdot (n/ad + Ny/v) \ = \ u + nvx/ad + Nz \\
& \iff & (nv/ad, (au + bv)N/dv, N) \, \big| \, u - (au + bv)n/ad^2 \\
& \iff & (v, N/v) \, \big| \, d^2 - n.
\end{eqnarray*}
\end{proof}

Now we are ready to prove one of our main theorems.

\begin{proof}[Proof of Theorem \ref{theorem_main1}]
We prove by computing the divisor of the function $\Phi_n^{\scalebox{0.6}{$\Gamma_0(N)$}}(j_N, j_N)$ on the compact Riemann surface $X_0(N)$ of genus zero. Since $\Phi_n^{\scalebox{0.6}{$\Gamma_0(N)$}}(j_N, j_N) \in \mathbb Z[j_N]$ by Lemma \ref{lemma_Phi} (2), it is holomorphic everywhere except the cusp $\infty$ of $X_0(N)$. Keeping this in mind, we show the following.

\begin{claim}
Both of
\[ \Phi_n^{\scalebox{0.6}{$\Gamma_0(N)$}}(j_N(\tau), \, j_N(\tau)) \qquad \mbox{and} \qquad \prod_{r \in \mathbb Z \atop |r| < 2 \sqrt{n}} H_{r^2 - 4n, N}(j_N(\tau)) \]
have the same divisor on $Y_0(N)$.
\end{claim}

\begin{proof}[Proof of Claim 1]
Note that $\Phi_n^{\scalebox{0.6}{$\Gamma_0(N)$}}(j_N(\tau), j_N(\tau))$ vanishes at $\Gamma_0(N)\tau_0 \in Y_0(N)$ if and only if there exists $\bar{\alpha} \in \overline{\mathrm{M}}_{n, N}^{\mathrm{ell}}$ such that $\alpha(\tau_0) = \tau_0$. We define an equivalence relation on $\overline{\mathrm{M}}_{n, N}^{\mathrm{ell}}(R)$ by
\[ \bar{\alpha} \sim \bar{\beta} \quad \mbox{if} \ \tau_\alpha = \tau_\beta, \ \overline{\Gamma}_0(N) \bar{\alpha} = \overline{\Gamma}_0(N) \bar{\beta} \]
and denote by $\langle \bar{\alpha} \rangle$ the equivalence class of $\bar{\alpha}$. Then we infer from the definitions of $\Phi_n^{\scalebox{0.6}{$\Gamma_0(N)$}}(j_N(\tau), j_N(\tau)) = \prod_{[\bar{\alpha}] \in \overline{\Gamma}_0(N) \backslash \overline{\mathrm{M}}_{n, N}} (j_N(\tau) - j_N(\alpha(\tau)))$ and $\overline{\mathrm{M}}_{n, N}^{\mathrm{ell}}(R)/\sim$ that both of
\[ \Phi_n^{\scalebox{0.6}{$\Gamma_0(N)$}}(j_N(\tau), j_N(\tau)) \qquad \mbox{and} \qquad \prod_{\langle \bar{\alpha} \rangle \in \overline{\mathrm{M}}_{n, N}^{\mathrm{ell}}(R) / \sim} \left(j_N(\tau) - j_N(\alpha(\tau))\right) \]
have the same divisor on $Y_0(N)$.

Now Lemma \ref{lemma_j_N} says that the contribution of $j_N(\tau) - j_N(\alpha(\tau))$ to the multiplicity of the zero $\Gamma_0(N) \tau_\alpha \in Y_0(N)$ of $\Phi_n^{\scalebox{0.6}{$\Gamma_0(N)$}}(j_N(\tau), j_N(\tau))$ is the same as the contribution of $j_N(\tau) - j_N(\tau_\alpha)$ to the multiplicity, which is in fact the order $1$. Hence the latter of the two can be replaced by
\[ \qquad \prod_{\langle \bar{\alpha} \rangle \in \overline{\mathrm{M}}_{n, N}^{\mathrm{ell}}(R) / \sim} \left(j_N(\tau) - j_N(\tau_\alpha)\right). \]
Since $\langle \bar{\alpha} \rangle = \{ \bar{\gamma} \bar{\alpha} \in \overline{\mathrm{M}}_{n, N}^{\mathrm{ell}}(R) \, | \, \bar{\gamma} \in \overline{\Gamma}_0(N)_{\tau_\alpha} \}$, we see that $|\langle \bar{\alpha} \rangle| = |\overline{\Gamma}_0(N)_{\tau_\alpha}|$ and deduce that
\begin{eqnarray*}
\prod_{\langle \bar{\alpha} \rangle \in \overline{\mathrm{M}}_{n, N}^{\mathrm{ell}}(R) / \sim} \left(j_N(\tau) - j_N(\tau_\alpha)\right) & = & \prod_{\langle \bar{\alpha} \rangle \in \overline{\mathrm{M}}_{n, N}^{\mathrm{ell}}(R) / \sim} \ \prod_{\bar{\beta} \in \langle \bar{\alpha} \rangle} \left(j_N(\tau) - j_N(\tau_\alpha)\right)^{1/|\overline{\Gamma}_0(N)_{\tau_\alpha}|} \\
& = & \prod_{\bar{\alpha} \in \overline{\mathrm{M}}_{n, N}^{\mathrm{ell}}(R)} (j_N(\tau) - j_N(\tau_\alpha))^{1/|\overline{\Gamma}_0(N)_{\tau_\alpha}|}.
\end{eqnarray*}
By Lemma \ref{lemma_q} we have $\tau_\alpha = \tau_{q(\alpha)}$ and $|\overline{\Gamma}_0(N)_{\tau_\alpha}| = |\overline{\Gamma}_0(N)_{q(\alpha)}| = \omega_{q(\alpha), N}$ and hence the proof is complete by means of Lemma \ref{lemma_qbar}.
\end{proof}

Now it remains to compute the divisor at the cusps.

\begin{claim}
Both of
\[ \Phi_n^{\scalebox{0.6}{$\Gamma_0(N)$}}(j_N(\tau), \, j_N(\tau)) \qquad \mbox{and} \qquad \prod_{u/v \in S'-\{ 1/N \}} \big(j_N(\tau) \, - \, j_N(u/v)\big)^{\nu_{u/v, n, N}} \]
have the same divisor on $X_0(N) - Y_0(N) - \{ \infty \}$.
\end{claim}

\begin{proof}[Proof of Claim 2]
Because $\Phi_n^{\scalebox{0.6}{$\Gamma_0(N)$}}(j_N(\tau), j_N(\tau))$ has a zero at $u/v$ if and only if $j_N(\tau) - j_N(\alpha_{a, b}(\tau))$ has a zero at $u/v$ for some $a, b$, we need to first compute the Fourier expansion of $j_N(\tau) - j_N(\alpha_{a, b}(\tau))$ at $u/v$ and then sum up their orders over all such $a, b$.

Since $(u, v) = 1$, there are $x, y \in \mathbb Z$ such that $ux - vy = 1$. Put $\rho = (\begin{smallmatrix} u & y \\ v & x \end{smallmatrix})$. Because $h := \frac{N}{v(v, N/v)}$ is the width of the cusp $u/v$ of $\Gamma_0(N)$, $q_h := e^{2 \pi i \tau / h}$ is a local parameter at $u/v \in X_0(N)$. So by Lemma \ref{lemma_c_1} we have the Fourier expansion of $j_N(\tau)$ at $u/v$ as
\[ j_N(\rho(\tau)) \ = \ j_N(u/v) \ + \ \sum_{m = 1}^\infty c_m q_h^m \]
with $c_1 \neq 0$.

Now we compute the Fourier expansion of $j_N(\alpha_{a, b}(\tau))$ at $u/v$. Observe that $j_N(\alpha_{a, b}(\tau))$ generates the function field $\mathbb C(X')$, where $X' = \Gamma' \backslash \mathbb H^*$ and $\Gamma' = \alpha_{a, b}^{-1} \Gamma_0(N) \alpha_{a, b}$. By definition, the width of the cusp $u/v$ of $\Gamma'$ is the smallest positive number $h'$ such that $\rho (\begin{smallmatrix} 1 & h' \\ 0 & 1 \end{smallmatrix}) \rho^{-1} \in \{ \pm 1 \} \cdot \Gamma'$. By a straightforward computation we obtain that
\begin{eqnarray*}
\rho (\begin{smallmatrix} 1 & h' \\ 0 & 1 \end{smallmatrix}) \rho^{-1} \, \in \, \{ \pm 1 \} \cdot \Gamma' & \iff & \begin{pmatrix}
1 - uvh' - bv^2h'/a & (au + bv)^2 h'/n \\
-n v^2 h' /a^2 & 1 + uvh' + bv^2h'/a
\end{pmatrix} \ \in \ \Gamma_0(N) \\
& \iff & h' \ \in \ \frac{a}{v(au + bv)} \mathbb Z \, \bigcap \, \frac{n}{(au + bv)^2} \mathbb Z \, \bigcap \, \frac{a^2 N}{nv^2} \mathbb Z \\
& \iff & h' \ \in \ \frac{\mathrm{lcm}(anv(au + bv), \, n^2v^2, \, a^2 (au + bv)^2 N)}{nv^2 (au + bv)^2} \mathbb Z \\
& \iff & h' \ \in \ \frac{N}{v(v, N/v)} \cdot \frac{n}{(au + bv, \, n/a)^2} \mathbb Z,
\end{eqnarray*}
from which we see that the width $h'$ of the cusp $u/v$ of $\Gamma'$ is given as
\[ h' \ = \ \frac{N}{v(v, N/v)} \cdot \frac{n}{(au + bv, \, n/a)^2}. \]
Thus the Fourier expansion of $j_N(\alpha_{a, b}(\tau))$ at the cusp $u/v \in X'$ is of the form
\[ j_N(\alpha_{a, b}(\rho(\tau))) \ = \ j_N(\alpha_{a, b}(u/v)) \ + \ \sum_{m = 1}^\infty d_m q_{h'}^m \]
with $d_1 \neq 0$. Here we have used Lemma \ref{lemma_c_1} again.

Because $\alpha_{a, b}(u/v)$ is equivalent to $u/v$ under $\Gamma_0(N)$, we obtain that
\[ j_N(\rho(\tau)) \ - \ j_N(\alpha_{a, b}(\rho(\tau))) \ = \ c_1 q_h \ - \ d_1 q_{h'} \ + \ \cdots. \]
Note that $h \neq h'$ because $n$ is not a perfect square. Hence all of $c_1 q_h$, $d_1 q_{h'}$, $c_1 q_h - d_1 q_{h'}$ are nonzero. Thus the lowest degree with respect to $q = e^{2 \pi i \tau}$ is $\mathrm{min} \{ 1/h, 1/h' \}$. Combining this with Lemma \ref{lemma_cusp} we derive that the order of $\Phi_n^{\scalebox{0.6}{$\Gamma_0(N)$}}(j_N(\tau), j_N(\tau)) \in \mathbb C(X_0(N))$ at the cusp $u/v$ is
\begin{eqnarray*}
& & h \cdot \sum_{a | n, \, 0 \leq b < n/a \atop n \equiv (au + bv, n/a)^2 \, \bmod {(v, N/v)}} \mathrm{min} \{ 1/h, \ 1/h' \} \\
& = & \sum_{a|n} \sum_{0 \leq b < n/a \atop n \equiv (b, n/a)^2 \bmod{(v, N/v)}} \min \left\{ 1, \ \frac{(b, n/a)^2}{n} \right\} \qquad \mbox{because } (v, n/a) \ = \ 1 \\
& = & \sum_{d|n} \sum_{0 \leq b < d \atop n \equiv (b, d)^2 \bmod{(v, N/v)}} \min \left\{ 1, \ \frac{(b, d)^2}{n} \right\}\\
& = & \sum_{e|n \atop n \equiv e^2 \bmod{(v, N/v)}} \sum_{d'|\frac{n}{e}} \sum_{0 \leq b' < d' \atop (b', d') = 1} \min \{ 1, \ e^2 / n \} \\
& = & \sum_{e|n \atop n \equiv e^2 \bmod{(v, N/v)}} \min \{ e, \ n/e \} \\
& = & \nu_{u/v, n, N}.
\end{eqnarray*}
Since $j_N(\tau) - j_N(u/v) \in \mathbb C(X_0(N))$ has a simple zero at $u/v$, both of $\Phi_n^{\scalebox{0.6}{$\Gamma_0(N)$}}(j_N(\tau), j_N(\tau))$ and $(j_N(\tau) - j_N(u/v))^{\nu_{u/v, n, N}}$ have the same order at $u/v$.
\end{proof}

We infer from Claims 1 and 2 that both of
\[ \Phi_n^{\scalebox{0.6}{$\Gamma_0(N)$}}(j_N(\tau), j_N(\tau)) \ \ \mbox{and} \ \prod_{r \in \mathbb Z \atop |r| < 2 \sqrt{n}} H_{r^2 - 4n, N}(j_N(\tau)) \times \prod_{s \in S \atop s \not \sim \infty} \big(j_N(\tau) \, - \, j_N(s)\big)^{\nu_{s, n, N}} \]
have the same zeros with the same multiplicities on $X_0(N) - \{ \infty \}$. So they must have the unique pole $\infty$ with the same multiplicity. In conclusion, they have the same divisor on $X_0(N)$ and hence differ only by a constant multiple, which is $(-1)^{\frac{1}{2}\sigma_0(n)}$ by Lemma \ref{lemma_Phi} (3).
\end{proof}

We further define the Hurwitz-Kronecker class number $H(D, N)$ and the trace $t(D, N)$ of singular moduli as
\begin{eqnarray*}
H(D, N) & = & \sum_{[Q] \in \mathrm{Q}_{D, N} / \Gamma_0(N)} \frac{1}{\omega_{Q, N}} \\
t(D, N) & = & \sum_{[Q] \in \mathrm{Q}_{D, N} / \Gamma_0(N)} \frac{1}{\omega_{Q, N}} j_N(\tau_Q).
\end{eqnarray*}

\begin{corollary}\label{corollary_main1}
With the assumptions as in Theorem \ref{theorem_main1}, we have
\begin{eqnarray*}
\textrm{(1)} & \hspace{-.5cm} & \sum_{r \in \mathbb Z \atop |r| < 2 \sqrt{n}} H(r^2 - 4n, N) \ + \ \sum_{s \in S} \nu_{s, n, N} \ = \ 2\sigma(n) \\
\textrm{(2)} & \hspace{-.5cm} & \sum_{r \in \mathbb Z \atop |r| < 2 \sqrt{n}} t(r^2 - 4n, N) \ + \ \sum_{s \in S \atop s \not \sim \infty} \nu_{s, n, N} \cdot j_N(s) \ = \ \left\{ \begin{array}{ll}
2 & \mbox{if $4n + 1$ is a perfect square,} \\
0 & \mbox{otherwise.}
\end{array} \right.
\end{eqnarray*}
\end{corollary}

\begin{proof}
(1) The degree of $\Phi_n^{\scalebox{0.6}{$\Gamma_0(N)$}}(X, X)$ in $X$ is $\sum_{d|n} \mathrm{max} \{d, n/d\}$ by Lemma \ref{lemma_Phi} (3) and the degree of the right-hand side in Theorem \ref{theorem_main1} is
\[ \sum_{r \in \mathbb Z \atop |r| < 2 \sqrt{n}} H(r^2 - 4n, N) \ + \ \sum_{s\in S \atop s \not\sim \infty} \nu_{s, n, N}, \]
whence we deduce our assertion with the help of
\[ \sum_{d|n} \mathrm{max}\{d, \, n/d\} \ + \ \sum_{d|n} \mathrm{min}\{d, n/d\} \ = \ 2 \sigma(n). \]

(2) We have
\begin{eqnarray*}
H_{D, N}(j_N(\tau)) & = & \prod_{[Q] \in Q_{D, N}/\Gamma_0(N)} (j_N(\tau) \ - \ j_N(\tau_Q))^{\frac{1}{\omega_{Q, N}}} \\
& = & \prod_{[Q] \in Q_{D, N}/\Gamma_0(N)} (q^{-1} \ - \ j_N(\tau_Q) \ + \ \mathrm{O}(q))^{\frac{1}{\omega_{Q, N}}} \\
& = & \prod_{[Q] \in Q_{D, N}/\Gamma_0(N)} q^{-\frac{1}{\omega_{Q, N}}} \Big(1 \ - \ \frac{1}{\omega_{Q, N}} j_N(\tau_Q)q \ + \ \mathrm{O}(q^2)\Big) \\
& = & q^{-H(D, N)} (1 \ - \ t(D, N)q \ + \ \mathrm{O}(q^2))
\end{eqnarray*}
and
\begin{eqnarray*}
\prod_{s \in S \atop s \not\sim \infty} (j_N(\tau) \ - \ j_N(s))^{\nu_{s, n, N}} & = & \prod_{s \in S \atop s \not\sim \infty} (q^{-1} \ - \ j_N(s) \ + \ \mathrm{O}(q))^{\nu_{s, n, N}} \\
& = & q^{-\sum_{s \in S \atop s \not\sim \infty} \nu_{s, n, N}} \Big(1 \ - \ \sum_{s \in S \atop s \not\sim \infty} \nu_{s, n, N} \cdot j_N(s) q \ + \ \mathrm{O}(q^2)\Big).
\end{eqnarray*}

On the other hand, we have
\begin{eqnarray*}
\Phi_n^{\scalebox{0.6}{$\Gamma_0(N)$}}(j_N(\tau), \, j_N(\tau)) & = & \prod_{a|n} \prod_{0 \leq b < n/a} \Big(j_N(\tau) \ - \ j_N\Big(\frac{a\tau + b}{n/a}\Big)\Big) \\
& = & \prod_{a|n} \prod_{0 \leq b < n/a} (q^{-1} \ - \ \zeta_n^{-ab} q_n^{-a^2} \ + \ \mathrm{O}(q_n)) \\
& = & \prod_{a|n} (q^{-n/a} \ - \ q^{-a})(1 \ + \ \mathrm{O}(q^{1 + \frac{1}{n}})),
\end{eqnarray*}
because $q^{-n/a} - q^{-a} = \prod_{0 \leq b < n/a}(q^{-1} - \zeta_n^{-ab} q_n^{-a^2})$. Note that
\[ (q^{-n/a} \ - \ q^{-a})(1 \ + \ \mathrm{O}(q^{1 + \frac{1}{n}})) \ = \ \pm q^{-\mathrm{max}\{ a, n/a \}} (1 \ - \ \varepsilon_a \cdot q \ + \ \mathrm{O}(q^{1 + \frac{1}{n}})), \]
where $\varepsilon_a$ is $1$ if $|a - n/a| = 1$ (which can happen if and only if $4n + 1$ is a perfect square) and $0$ otherwise. Thus we see that
\[ \Phi_n^{\scalebox{0.6}{$\Gamma_0(N)$}}(j_N(\tau), \, j_N(\tau)) \ = \ (-1)^{\sum_{a|n, \, a \geq n/a} 1} \, q^{-\sum_{a|n} \max\{ a, n/a \}} (1 \ - \ (\sum_{a|n} \varepsilon_a) q \ + \ \mathrm{O}(q^{1 + \frac{1}{n}})). \]

We may recover the identity (1) by comparing the lowest degree in $q$ of both sides in Theorem \ref{theorem_main1}. A new identity (2) can further be obtained by comparing the next lowest degree with the help of
\[ \sum_{a|n} \varepsilon_a \ = \ \left\{ \begin{array}{ll}
2 & \mbox{if $4n + 1$ is a perfect square,} \\
0 & \mbox{otherwise.}
\end{array} \right. \]
\end{proof}

If $(n, N) = 1$ and if $n$ is a perfect square, then $X - Y$ divides $\Phi_n^{\scalebox{0.6}{$\Gamma_0(N)$}}(X, Y)$, so $\Phi_n^{\scalebox{0.6}{$\Gamma_0(N)$}}(X, X)$ becomes the constant polynomial $0$. Thus in that case it is natural to deal with $\left.\frac{\Phi_n^{\scalebox{0.6}{$\Gamma_0(N)$}}(X, Y)}{\Phi_1^{\scalebox{0.6}{$\Gamma_0(N)$}}(X, Y)}\right|_{Y = X}$ like Zagier did.

\begin{theorem}\label{theorem_main2}
Assume that the genus of $\Gamma_0(N)$ is zero and that $n$ is a positive integer relatively prime to $N$. Assume further that $n$ is a perfect square. Then we have
\begin{eqnarray*}
\textrm{(1)} & \hspace{-.5cm} & \hspace{.3cm} \left. \frac{\Phi_n^{\scalebox{0.6}{$\Gamma_0(N)$}}(X, Y)}{\Phi_1^{\scalebox{0.6}{$\Gamma_0(N)$}}(X, Y)}\right|_{Y = X} \ = \ (-1)^{\frac{1}{2} (\sigma_0(n)-1)} \sqrt{n} \ \times \ \frac{\prod_{r \in \mathbb Z \atop |r| < 2 \sqrt{n}} H_{r^2 - 4n, N}(X)}{\prod_{r \in \mathbb Z \atop |r| < 2} H_{r^2 - 4, N}(X)} \ \times \ \prod_{s \in S \atop s \not \sim \infty} \big(X \, - \, j_N(s)\big)^{\nu_{s, n, N} - 1} \\
\textrm{(2)} & \hspace{-.5cm} & \sum_{r \in \mathbb Z \atop |r| < 2 \sqrt{n}} H(r^2 - 4n, N) \ + \ \sum_{s \in S} \nu_{s, n, N} \ = \ 2 \sigma(n) \ + \ \sum_{r \in \mathbb Z \atop |r| < 2} H(r^2 - 4, N) \ + \ \sum_{d|N} \varphi((d, \, N/d)) \ - \ 2 \\
\textrm{(3)} & \hspace{-.5cm} & \sum_{r \in \mathbb Z \atop |r| < 2 \sqrt{n}} t(r^2 - 4n, N) \ + \ \sum_{s \in S \atop s \not \sim \infty} (\nu_{s, n, N} - 1)j_N(s) \ = \ \sum_{r \in \mathbb Z \atop |r| < 2} t(r^2 - 4, N).
\end{eqnarray*}
Here, $S$ denotes the set of all inequivalent cusps of $\Gamma_0(N)$.
\end{theorem}

\begin{proof}
We have
\[ \frac{\Phi_n^{\scalebox{0.6}{$\Gamma_0(N)$}}(X, j_N)}{\Phi_1^{\scalebox{0.6}{$\Gamma_0(N)$}}(X, j_N)} \ = \ \prod_{a|n, \, 0 \leq b < n/a \atop (a, b) \neq (\sqrt{n}, 0)} (X \ - \ j_N(\alpha_{a, b}(\tau))). \]

\begin{claim'}
Both of
\[ \prod_{a|n, \, 0 \leq b < n/a \atop (a, b) \neq (\sqrt{n}, 0)} (j_N(\tau) \ - \ j_N(\alpha_{a, b}(\tau))) \qquad \mbox{and} \qquad \frac{\prod_{r \in \mathbb Z \atop |r| < 2 \sqrt{n}} H_{r^2 - 4n, N}(j_N(\tau))}{\prod_{r \in \mathbb Z \atop |r| < 2} H_{r^2 - 4, N}(j_N(\tau))} \]
have the same divisor on $Y_0(N)$.
\end{claim'}

\begin{proof}[Proof of Claim 1]
We need to exclude the fixed points $\tau_\alpha \in R$ of all elliptic elements $\bar{\alpha} \in [\bar{\alpha}_{\sqrt{n}, 0}] = \sqrt{n} \, \overline{\Gamma}_0(N)$. We infer from the proof of Claim 1 in Theorem \ref{theorem_main1} that both of
\[ \prod_{a|n, \, 0 \leq b < n/a \atop (a, b) \neq (\sqrt{n}, 0)} (j_N(\tau) \ - \ j_N(\alpha_{a, b}(\tau))) \quad \mbox{and} \quad \prod_{\bar{\alpha} \in \overline{\mathrm{M}}_{n, N}^{\mathrm{ell}}(R) - \sqrt{n} \, \overline{\Gamma}^{\mathrm{ell}}_0(N)(R)} (j_N(\tau) - j_N(\tau_\alpha))^{1/|\overline{\Gamma}_0(N)_{\tau_\alpha}|} \]
have the same divisor on $Y_0(N)$, where $\overline{\Gamma}^{\mathrm{ell}}_0(N)(R)$ denotes the subset of $\overline{\Gamma}_0(N)$ consisting of all elliptic elements whose fixed points on $\mathbb H$ are contained in $R$. Now notice that the fixed points of all elliptic elements of $\sqrt{n} \, \overline{\Gamma}^{\mathrm{ell}}_0(N)(R)$ are the same as the fixed points of those of $\overline{\Gamma}^{\mathrm{ell}}_0(N)(R)$. Since $\overline{\Gamma}^{\mathrm{ell}}_0(N)(R) = \overline{\mathrm{M}}_{1, N}^{\mathrm{ell}}(R)$, we see by Lemma \ref{lemma_qbar} that both of
\[ \prod_{\bar{\alpha} \in \sqrt{n} \, \overline{\Gamma}^{\mathrm{ell}}_0(N)(R)} (j_N(\tau) - j_N(\tau_\alpha))^{1/|\overline{\Gamma}_0(N)_{\tau_\alpha}|} \qquad \mbox{and} \qquad \prod_{r \in \mathbb Z \atop |r| < 2} H_{r^2 - 4, N}(j_N(\tau)) \]
have the same divisor on $Y_0(N)$, from which our assertion follows.
\end{proof}

Now we compute their divisors at cusps.

\begin{claim'}
Both of
\[ \prod_{a|n, \, 0 \leq b < n/a \atop (a, b) \neq (\sqrt{n}, 0)} (j_N(\tau) \ - \ j_N(\alpha_{a, b}(\tau))) \qquad \mbox{and} \qquad \prod_{u/v \in S'-\{ 1/N \}} \big(j_N(\tau) \, - \, j_N(u/v)\big)^{\nu_{u/v, n, N} - 1} \]
have the same divisor on $X_0(N) - Y_0(N) - \{ \infty \}$.
\end{claim'}

\begin{proof}[Proof of Claim 2]
We employ the same notation as in the proof of Claim 2 in Theorem \ref{theorem_main1}. Suppose that $\alpha_{a, b}(u/v) \sim u/v$ under $\Gamma_0(N)$. Then $\alpha_{a, b}(u/v) = \gamma(u/v)$ for some $\gamma \in \Gamma_0(N)$, and so $\rho^{-1} \gamma^{-1} \alpha_{a, b} \rho$ must be of the form $(\begin{smallmatrix} A & B \\ 0 & D \end{smallmatrix})$ because it sends $\infty$ to $\infty$. Thus we see that
\begin{eqnarray*}
(j_N \circ \alpha_{a, b} \circ \rho)(\tau) & = & (j_N \circ \rho)\Big(\frac{A\tau + B}{D}\Big) \\
& = & j_N(u/v) \ + \ \sum_{m = 1}^\infty c_m \zeta_h^{Bm / D} q_h^{Am / D},
\end{eqnarray*}
from which we infer that
\[ d_1 q_{h'} \ = \ c_1 \zeta_h^{B/D} q_h^{A/D}. \]

Assume that $d_1 q_{h'} = c_1 q_h$. Then we have $A = D$ and $B = Dhk$ for some $k \in \mathbb Z$. Since $AD = n$, we see that $A = D = \sqrt{n}$ and $B = \sqrt{n}hk$, whence
\begin{eqnarray*}
\alpha_{a, b} \ = \ \sqrt{n} \gamma \rho (\begin{smallmatrix} 1 & hk \\ 0 & 1 \end{smallmatrix}) \rho^{-1} & \Longrightarrow & \sqrt{n} \, \big| \, a, \ \sqrt{n} \, \big| \, b, \ \sqrt{n} \, \big| \, n/a \\
& \Longrightarrow & a = \sqrt{n}, \ b = 0.
\end{eqnarray*}
However, the case that $a = \sqrt{n}$ and $b = 0$ does not concern us. So we have shown that $c_1q_h - d_1 q_{h'} \neq 0$, and hence its order of zeros at $u/v$ is
\begin{eqnarray*}
& & h \cdot \sum_{a | n, \, 0 \leq b < n/a, \, (a, b) \neq (\sqrt{n}, 0) \atop n \equiv (au + bv, n/a)^2 \, \bmod {(v, N/v)}} \mathrm{min} \{ 1/h, \ 1/h' \} \\
& = & h \cdot \sum_{a | n, \, 0 \leq b < n/a \atop n \equiv (au + bv, n/a)^2 \, \bmod {(v, N/v)}} \mathrm{min} \{ 1/h, \ 1/h' \} \ - \ 1 \\
& = & \nu_{u/v, n, N} \ - \ 1.
\end{eqnarray*}
\end{proof}

This shows the identity (1). The identity (2) can be obtained by equating the degrees of both sides of (1) with the help of $|S| = \sum_{d|N} \varphi((d, N/d))$. Finally we prove (3). By the same computation as in the proof of Corollary \ref{corollary_main1} (2), we have
\begin{eqnarray*}
& & \prod_{a|n, \, 0 \leq b < n/a \atop (a, b) \neq (\sqrt{n}, 0)} (j_N(\tau) \ - \ j_N(\alpha_{a, b}(\tau))) \\
& = & (-1)^{\frac{1}{2}(\sigma_0(n) - 1)} \sqrt{n} \, q^{-\sum_{a|n} \mathrm{max} \{ a, n/a \} + 1} \big( 1 \ - \ \big( \sum_{a|n \atop a \neq \sqrt{n}} \varepsilon_a \big)q \ + \ \mathrm{O}(q^{1 + \frac{1}{n}}) \big) \\
& = & (-1)^{\frac{1}{2}(\sigma_0(n) - 1)} \sqrt{n} \, q^{-\sum_{a|n} \mathrm{max} \{ a, n/a \} + 1} \big( 1 \ + \ \mathrm{O}(q^{1 + \frac{1}{n}}) \big)
\end{eqnarray*}
Here we have used the fact that $\varepsilon_a = 0$ because it is impossible that both $n$ and $4n + 1$ are perfect squares at the same time. On the other hand, we compute
\begin{eqnarray*}
\frac{\prod_{|r| < 2 \sqrt{n}} H_{r^2 - 4n, N}(j_N(\tau))}{\prod_{|r| < 2} H_{r^2 - 4, N}(j_N(\tau))} & = & q^{-\sum_{|r|<2\sqrt{n}} H(r^2 - 4n, N) + \sum_{|r|<2} H(r^2 - 4, N)} \\
& & \ \times \ \big( 1 - \big( \sum_{|r| < 2\sqrt{n}} t(r^2 - 4n, N) - \sum_{|r| < 2} t(r^2 - 4, N) \big)q + \mathrm{O}(q^2) \big)
\end{eqnarray*}
and
\begin{eqnarray*}
\prod_{s \in S \atop s \not \sim \infty} \big(j_N(\tau) \, - \, j_N(s)\big)^{\nu_{s, n, N} - 1} & = & q^{-\sum_{s \in S \atop s \not\sim \infty} (\nu_{s, n, N} - 1)} \big( 1 - \sum_{s \in S \atop s \not\sim \infty} (\nu_{s, n, N} - 1)j_N(s) q + \mathrm{O}(q^2) \big).
\end{eqnarray*}
Comparing their coefficients, we obtain our assertion.
\end{proof}

\section{The $\Gamma$-equivalence of binary quadratic forms}

We denote by $\mathrm{QF}$ the set of all primitive positive-definite binary quadratic forms $Q(x, y) = ax^2 + bxy + cy^2$, and denote by $\mathrm{QT}$ the set of all points $\tau \in \mathbb H$ such that $[\mathbb Q(\tau) : \mathbb Q] = 2$. Then any congruence subgroup $\Gamma$ of $\mathrm{SL}_2(\mathbb Z)$ acts on both $\mathrm{QF}$ and $\mathrm{QT}$ in a natural way as
\begin{eqnarray*}
(Q \cdot \gamma)(x, y) & = & (x \ y) \tensor*[^t]\gamma{} A_Q \gamma \left(\begin{smallmatrix} x \\ y \end{smallmatrix}\right), \\
\gamma(\tau) & = & \frac{a\tau + b}{c \tau + d},
\end{eqnarray*}
where $Q \in \mathrm{QF}$, $\gamma = \left(\begin{smallmatrix} a & b \\ c & d \end{smallmatrix}\right) \in \Gamma$, $\tau \in \mathrm{QT}$, and $A_Q$ is a symmetric matrix associated with $Q$. By definition, a form $Q'$ is $\Gamma$-equivalent to a form $Q$ if and only if $Q' = Q\cdot \gamma$ for some $\gamma \in \Gamma$.

Given a form $Q \in \mathrm{QF}$, let $\tau_Q$ denote the unique solution of $Q(x, 1) = 0$ in $\mathbb H$. Then it is not difficult to see that the map defined by $Q \mapsto \tau_Q$ is a one-to-one correspondence between $\mathrm{QF}$ and $\mathrm{QT}$, and moreover
\[ \tau_{Q \cdot \gamma} \, = \, \gamma^{-1}(\tau_Q). \]

\begin{definition}
Let $\mathfrak F_\Gamma$ be a fundamental region for a congruence subgroup $\Gamma$.

(1) A form $Q \in \mathrm{QF}$ is said to be $\Gamma$-reduced if $\tau_Q \in \mathfrak F_\Gamma$. The set of all $\Gamma$-reduced forms in $\mathrm{QF}$ is denoted by $\Gamma\text{-}\mathrm{RF}$.

(2) Given a negative integer $D \equiv 0, 1 \pmod 4$, we denote by $\mathrm{QF}(D)$ the set of all forms in $\mathrm{QF}$ of discriminant $D$, and denote by $\Gamma\text{-}\mathrm{RF}(D)$ the set of all $\Gamma$-reduced forms in $\mathrm{QF}(D)$.
\end{definition}

It is worthwhile to remark that the definition of $\Gamma$-reduced forms depends on the choice of the fundamental region $\mathfrak F_\Gamma$.

\begin{example}
As is well known, a fundamental region for $\mathrm{SL}_2(\mathbb Z)$ is
\begin{eqnarray*}
\mathfrak F_{\mathrm{SL}_2(\mathbb Z)} & = & \big\{ \tau \in \mathbb H \, \big| \, |\mathrm{Re}(\tau)| \leq 1/2, \ |\tau| \geq 1, \\
& & \qquad \quad \ \ \big( |\mathrm{Re}(\tau)| = 1/2 \mbox{ or } |\tau| = 1 \big) \Longrightarrow \mathrm{Re}(\tau) \leq 0 \big\}.
\end{eqnarray*}
With respect to this choice of $\mathfrak F_{\mathrm{SL}_2(\mathbb Z)}$, a form $ax^2 + bxy + cy^2 \in \mathrm{QF}$ is $\mathrm{SL}_2(\mathbb Z)$-reduced if and only if it satisfies

(1) $|b| \leq a \leq c$

(2) $(|b| = a \mbox{ or } a = c) \Longrightarrow b \geq 0$.\\
Notice that this coincides with the usual definition of reduced forms for the proper equivalence. This relation between reduced forms and fundamental domains can also be found in \cite[Section 5.3.1]{Coh} and \cite[Exercise 11.4]{Cox}.
\end{example}

\begin{example}\label{Example - 2, 3, 4}
For $N = 2, 3, 4$ it is known that
\begin{eqnarray*}
\mathfrak F_{\Gamma_0(N)} & = & \big\{ \tau \in \mathbb H \, \big| \, |\mathrm{Re}(\tau)| \leq 1/2, \ |\tau - 1/N| \geq 1/N, \ |\tau + 1/N| \geq 1/N, \\
& & \qquad \quad \ \ \big( |\mathrm{Re}(\tau)| = 1/2 \mbox{ or } |\tau - 1/N| = 1/N \mbox{ or } |\tau + 1/N| = 1/N \big) \Longrightarrow \mathrm{Re}(\tau) \leq 0 \big\}
\end{eqnarray*}
(see also the remark at the end of Section 4). Therefore $ax^2 + bxy + cy^2$ is defined to be $\Gamma_0(N)$-reduced if

(1) $|b| \leq a$, $|b| \leq Nc$

(2) $(|b| = a \mbox{ or } |b| = Nc) \Longrightarrow b \geq 0$.\\
Notice that the condition (1) implies the finiteness of $\Gamma_0(N)\text{-}\mathrm{RF}(D)$. More precisely, when $N = 2$ or $3$, the finiteness follows immediately from the fact that
\begin{eqnarray*}
\frac{4-N}{N}b^2 \leq 4ac - b^2 = -D & \Longrightarrow & |b| \leq \sqrt{\frac{-ND}{4-N}}, \\
4ac = b^2 - D \leq \frac{-4D}{4 - N} & \Longrightarrow & ac \leq \frac{-D}{4 - N},
\end{eqnarray*}
since $b^2 \leq Nac$. For $N = 4$, we consider the following three cases:

\medskip

\underline{Case 1} $a > 4c$

We see that $-D = 4ac - b^2 \geq 4ac - 16c^2 = 4c(a - 4c) \geq 4c$. Hence there are at most finitely many triples $(a, b, c)$ such that $D = b^2 - 4ac$, $c \leq -D/4$, and $|b| \leq 4c$.

\medskip

\underline{Case 2} $a < 4c$

Similarly, the finiteness follows from $-D = 4ac - b^2 \geq 4ac - a^2 = a(4c - a) \geq a$.

\medskip

\underline{Case 3} $a = 4c$

It follows from $-D = (a-b)(a+b)$.

\medskip

Using these, we easily compute that
\begin{eqnarray*}
\Gamma_0(2)\text{-}\mathrm{RF}(-3) & = & \{ x^2 + xy + y^2 \}, \\
\Gamma_0(2)\text{-}\mathrm{RF}(-4) & = & \{ x^2 + y^2, \, 2x^2 + 2xy + y^2 \}, \\
\Gamma_0(2)\text{-}\mathrm{RF}(-7) & = & \{ x^2 + xy + 2y^2, \, 2x^2 \pm xy + y^2 \}, \\
\Gamma_0(2)\text{-}\mathrm{RF}(-8) & = & \{ x^2 + 2y^2, \, 2x^2 + y^2, \, 3x^2 + 2xy + y^2 \}
\end{eqnarray*}
and
\begin{eqnarray*}
\Gamma_0(3)\text{-}\mathrm{RF}(-3) & = & \{ x^2 + xy + y^2, \, 3x^2 + 3xy + y^2 \}, \\
\Gamma_0(3)\text{-}\mathrm{RF}(-4) & = & \{ x^2 + y^2, \, 2x^2 + 2xy + y^2 \}, \\
\Gamma_0(3)\text{-}\mathrm{RF}(-7) & = & \{ x^2 + xy + 2y^2, \, 2x^2 \pm xy + y^2, \, 4x^2 + 3xy + y^2 \}, \\
\Gamma_0(3)\text{-}\mathrm{RF}(-8) & = & \{ x^2 + 2y^2, \, 2x^2 + y^2, \, 3x^2 \pm 2xy + y^2 \}
\end{eqnarray*}
and
\begin{eqnarray*}
\Gamma_0(4)\text{-}\mathrm{RF}(-3) & = & \{ x^2 + xy + y^2, \, 3x^2 + 3xy + y^2 \}, \\
\Gamma_0(4)\text{-}\mathrm{RF}(-4) & = & \{ x^2 + y^2, \, 2x^2 + 2xy + y^2, \, 5x^2 + 4xy + y^2 \}, \\
\Gamma_0(4)\text{-}\mathrm{RF}(-7) & = & \{ x^2 + xy + 2y^2, \, 2x^2 \pm xy + y^2, \, 4x^2 \pm 3xy + y^2, \, 7x^2 + 7xy + 2y^2 \}, \\
\Gamma_0(4)\text{-}\mathrm{RF}(-8) & = & \{ x^2 + 2y^2, \, 2x^2 + y^2, \, 3x^2 \pm 2xy + y^2, \, 6x^2 + 4xy + y^2, \, 9x^2 + 8xy + 2y^2 \}.
\end{eqnarray*}
\end{example}

\begin{example}\label{Example - p}
Let $p \geq 5$ be a prime number. Then a form $ax^2 + bxy + cy^2 \in \mathrm{QF}$ is defined to be $\Gamma_0(p)$-reduced if it satisfies the following properties:

\begin{enumerate}
\item $|b| \leq a$

\item $|b| \leq \frac{p^2c + (k^2 - 1)a}{p \cdot |k|}$ for all $k \in S_p$

\item $(|b| = a \mbox{ or } |b| = pc) \Longrightarrow b \geq 0$

\item $b = - \frac{p^2c + (k^2 - 1)a}{pk}$ with $k \in E_p^{(2)}$ $\Longrightarrow$ $b \geq -\frac{2ka}{p}$

\item $b = - \frac{p^2c + (k^2 - 1)a}{pk}$ with $k \in S_p - (\{ \pm 1 \} \cup E_p^{(2)})$ $\Longrightarrow$ $b \geq -\frac{(2k_{(2)} + 1)a}{p}$

\item $b^2 - 4ac = - \frac{3a^2}{p^2}$ $\Longrightarrow$ $b \neq \frac{(1 - 2k)a}{p}$ for all $k \in S_p - (\{ 1 \} \cup E_p^{(3)})$ with $k \neq k_{(3)}$,
\end{enumerate}
where the definitions of $S_p$, $E_p^{(2)}$, $k_{(2)}$, $E_p^{(3)}$, and $k_{(3)}$ are given in Section 4. The definition of $\Gamma_0(p)$-reduction above corresponds to the fundamental region for $\Gamma_0(p)$ described explicitly in Section 4.

We can see that the conditions (1) and (2) are good enough to imply the finiteness of $\Gamma_0(p)\text{-}\mathrm{RF}(D)$. Note that
\begin{eqnarray*}
& & |b| \leq \frac{p^2c + (k^2 - 1)a}{p \cdot |k|} \ \mbox{ for all } k \in S_p \\
& \Longrightarrow & \big|\tau_Q - k/p\big| \geq 1/p \ \mbox{ for all } k \in S_p, \mbox{ where } \tau_Q = (-b + \sqrt{D})/2a \\
& \Longrightarrow & \mathrm{Im}(\tau_Q) \geq \sqrt{3}/2p \mbox{ or } \big(|\mathrm{Re}(\tau_Q)| \leq 1/2p, \ |\tau_Q - 1/p| \geq 1/p, \ |\tau_Q + 1/p| \geq 1/p \big) \\
& \Longrightarrow & a \leq \sqrt{-D/3} \cdot p \mbox{ or } \big(|b| \leq a/p, \ |b| \leq pc \big).
\end{eqnarray*}

Fix a negative discriminant $D$.

\medskip

\underline{Case 1} $|b| \leq a$, $a \leq \sqrt{-D/3} \cdot p$

Clearly, there are at most finitely many triples $(a, b, c)$.

\medskip

\underline{Case 2} $|b| \leq a/p$, $|b| \leq pc$

Since $b^2 = |b| \cdot |b| \leq \frac{a}{p} \cdot pc = ac$, the finiteness follows immediately from
\begin{eqnarray*}
3b^2 \leq 4ac - b^2 = -D & \Longrightarrow & |b| \leq \sqrt{-D/3}, \\
4ac = b^2 - D \leq -4D/3 & \Longrightarrow & ac \leq -D/3.
\end{eqnarray*}

According to our definition of $\Gamma_0(p)$-reduced forms above, a form $ax^2 + bxy + cy^2 \in \mathrm{QF}$ is $\Gamma_0(5)$-reduced if and only if

\begin{enumerate}
\item $|b| \leq a$

\item $|b| \leq 5c$

\item $(|b| = a \mbox{ or } |b| = 5c) \Longrightarrow b \geq 0$

\item $|b| \leq \frac{3a + 25c}{10}$

\item $b = \frac{3a + 25c}{10}$ $\Longrightarrow$ $a \leq 5c$

\item $b = -\frac{3a + 25c}{10}$ $\Longrightarrow$ $a \geq 5c$

\item $b^2 - 4ac = - \frac{3a^2}{25}$ $\Longrightarrow$ $b \neq \pm \frac{3a}{5}$.
\end{enumerate}
Under this definition, we easily compute that
\begin{eqnarray*}
\Gamma_0(5)\text{-}\mathrm{RF}(-3) & = & \{ x^2 + xy + y^2, \, 3x^2 + 3xy + y^2 \}, \\
\Gamma_0(5)\text{-}\mathrm{RF}(-4) & = & \{ x^2 + y^2, \, 2x^2 + 2xy + y^2, \, 5x^2 \pm 4xy + y^2 \}, \\
\Gamma_0(5)\text{-}\mathrm{RF}(-7) & = & \{ x^2 + xy + 2y^2, \, 2x^2 \pm xy + y^2, \, 4x^2 \pm 3xy + y^2, \, 7x^2 + 7xy + 2y^2 \}, \\
\Gamma_0(5)\text{-}\mathrm{RF}(-8) & = & \{ x^2 + 2y^2, \, 2x^2 + y^2, \, 3x^2 \pm 2 xy + y^2, \, 6x^2 \pm 4xy + y^2 \}. \\
\end{eqnarray*}
\end{example}

In what follows, we prove the finiteness of $\Gamma$-reduced forms of a given discriminant for any congruence subgroup $\Gamma$ of $\mathrm{SL}_2(\mathbb Z)$.

\begin{theorem}\label{Theorem - reduced}
Let $\Gamma\text{-}\mathrm{RF}(D)$ be the set of all $\Gamma$-reduced forms of discriminant $D$ with respect to $\mathfrak F_\Gamma$. Then we have the following.

(1) Any two distinct $\Gamma$-reduced forms in $\Gamma\text{-}\mathrm{RF}(D)$ are not $\Gamma$-equivalent.

(2) Every form $Q \in \mathrm{QF}(D)$ is $\Gamma$-equivalent to a unique $\Gamma$-reduced form in $\Gamma\text{-}\mathrm{RF}(D)$.

(3) The cardinality of $\Gamma\text{-}\mathrm{RF}(D)$ is independent of the choice of $\mathfrak F_\Gamma$.

(4) The cardinality $|\Gamma\text{-}\mathrm{RF}(D)|$ is finite.
\end{theorem}

\begin{proof}
(1) Suppose that $Q, Q' \in \Gamma\text{-}\mathrm{RF}(D)$ are $\Gamma$-equivalent to each other. Since $Q' = Q \cdot \gamma$ for some $\gamma \in \Gamma$, we have $\tau_{Q'} = \gamma^{-1}(\tau_Q)$. This says that $\tau_Q, \tau_{Q'} \in \mathfrak F_\Gamma$ are in the same $\Gamma$-orbit, from which we deduce that $\tau_Q = \tau_{Q'}$. Appealing to the one-to-one correspondence between $\mathrm{QF}$ and $\mathrm{QT}$, it follows that $Q =Q'$.

(2) Given a form $Q \in \mathrm{QF}(D)$, there exists $\gamma \in \Gamma$ such that $\tau_{Q \cdot \gamma} = \gamma^{-1}(\tau_Q)$ is contained in $\mathfrak F_\Gamma$. By definition, $Q$ is $\Gamma$-equivalent to $Q \cdot \gamma \in \Gamma\text{-}\mathrm{RF}(D)$. The uniqueness is an immediate consequence of (1).

(3) Let $\Gamma\text{-}\mathrm{RF}(D)$ and $\Gamma\text{-}\mathrm{RF}(D)'$ be the sets of $\Gamma$-reduced forms in $\mathrm{QF}(D)$ with respect to $\mathfrak F_\Gamma$ and $\mathfrak F_\Gamma'$, respectively. According to (2), there is a function $f:\Gamma\text{-}\mathrm{RF}(D) \rightarrow \Gamma\text{-}\mathrm{RF}(D)'$ by mapping $Q$ to a unique $Q'$ in $\Gamma\text{-}\mathrm{RF}(D)'$ which is $\Gamma$-equivalent to $Q$. Note that $f$ is injective by (1). For if $f(Q_1) = f(Q_2)$, then $Q_1, Q_2 \in \Gamma\text{-}\mathrm{RF}(D)$ are $\Gamma$-equivalent, so we have $Q_1 = Q_2$ by (1). Now the symmetry implies our desired assertion.

(4) Let $\gamma_1, \ldots, \gamma_n$ be all the distinct left coset representatives of $\Gamma$ in $\mathrm{SL}_2(\mathbb Z)$, where $n = [\mathrm{SL}_2(\mathbb Z) : \Gamma]$. Then we can choose $\mathfrak F_\Gamma$ so that $\mathfrak F_\Gamma$ is contained in
\[ \bigcup_{i=1}^n \gamma_i^{-1}(\mathfrak F_{\mathrm{SL}_2(\mathbb Z)}). \]
Under this choice of $\mathfrak F_\Gamma$ we have
\begin{eqnarray*}
\Gamma\text{-}\mathrm{RF}(D) & = & \{ Q \in \mathrm{QF}(D) \, | \, \tau_Q \in \mathfrak F_\Gamma \} \\
& \subset & \{ Q \in \mathrm{QF}(D) \, | \, \gamma_i(\tau_Q) \in \mathfrak F_{\mathrm{SL}_2(\mathbb Z)} \mbox{ for some } i = 1, \ldots, n \} \\
& = & \{ Q \in \mathrm{QF}(D) \, | \, Q \cdot \gamma_i^{-1} \mbox{ is $\mathrm{SL}_2(\mathbb Z)$-reduced for some } i = 1, \ldots, n \} \\
& = & \bigcup_{i=1}^n \{ Q' \cdot \gamma_i \in \mathrm{QF}(D) \, | \, Q' \in \mathrm{SL}_2(\mathbb Z)\text{-}\mathrm{RF}(D) \}.
\end{eqnarray*}
Because $|\mathrm{SL}_2(\mathbb Z)\text{-}\mathrm{RF}(D)|$ is finite, we can derive that $|\Gamma\text{-}\mathrm{RF}(D)|$ is also finite.
\end{proof}

\section{The fundamental region for $\Gamma_0(p)$}

Throughout this section, let $p \geq 5$ denote a prime number. The purpose of this section is to construct the fundamental region $\mathfrak F_{\Gamma_0(p)}$ for $\Gamma_0(p)$ which satisfies the following properties:

\begin{enumerate}
\item $|\mathrm{Re}(\tau)| \leq 1/2$ for all $\tau \in \mathfrak F_{\Gamma_0(p)}$

\item the imaginary part of $\tau \in \mathfrak F_{\Gamma_0(p)}$ is maximal in the orbit $\Gamma_0(p)\tau$

\item the real part of $\tau \in \mathfrak F_{\Gamma_0(p)}$ is minimal among the points $\tau' \in \Gamma_0(p)\tau$ with $|\mathrm{Re}(\tau')| \leq 1/2$

\item for any $\tau \in \mathbb H$, there exists a unique point $\tau' \in \mathfrak F_{\Gamma_0(p)}$ such that $\tau' \in \Gamma_0(p) \tau$.
\end{enumerate}

Given a congruence subgroup, we can explicitly construct its open fundamental region consisting of points whose imaginary part is maximal in its orbit. This can be done by applying the method given in Ferenbaugh's paper \cite[Section 3]{Fer}. For a matrix $\gamma = (\begin{smallmatrix} a & b \\ c & d \end{smallmatrix})$, we write $a_\gamma := a$, $b_\gamma := b$, and so on. As in his paper, let
\begin{eqnarray*}
\mathrm{arc}(\gamma) & = & \{ \tau \in \mathbb H \, \big| \, |\tau - a_\gamma/c_\gamma| = 1/|c_\gamma| \}, \\
\mathrm{inside}(\gamma) & = & \{ \tau \in \mathbb H \, \big| \, |\tau - a_\gamma/c_\gamma| < 1/|c_\gamma| \}, \\
\mathrm{outside}(\gamma) & = & \{ \tau \in \mathbb H \, \big| \, |\tau - a_\gamma/c_\gamma| > 1/|c_\gamma| \}
\end{eqnarray*}
for any $\gamma \in \mathrm{SL}_2(\mathbb Z)$ with $c_\gamma \neq 0$. If $\gamma = \pm(\begin{smallmatrix} 1 & b \\ 0 & 1 \end{smallmatrix})$ with $b > 0$, then we define
\begin{eqnarray*}
\mathrm{arc}(\gamma) & = & \{ \tau \in \mathbb H \, \big| \, \mathrm{Re}(\tau) = b/2 \}, \\
\mathrm{inside}(\gamma) & = & \{ \tau \in \mathbb H \, \big| \, \mathrm{Re}(\tau) > b/2 \}, \\
\mathrm{outside}(\gamma) & = & \{ \tau \in \mathbb H \, \big| \, \mathrm{Re}(\tau) < b/2 \}.
\end{eqnarray*}
On the other hand, if $\gamma = \pm(\begin{smallmatrix} 1 & b \\ 0 & 1 \end{smallmatrix})$ with $b < 0$, then we define $\mathrm{arc}(\gamma)$ in the same way and reverse the definitions of $\mathrm{inside}(\gamma)$ and $\mathrm{outside}(\gamma)$.

An open fundamental region for a congruence subgroup $\Gamma$ of $\mathrm{SL}_2(\mathbb Z)$ is defined to be an open subset $R$ of $\mathbb H$ with the properties:
\begin{enumerate}
\item there do not exist $\gamma \in \Gamma$ and $\tau, \tau' \in R$ such that $\tau \neq \tau'$ and $\tau = \gamma(\tau')$

\item for any $\tau \in \mathbb H$, there exists $\gamma \in \Gamma$ such that $\gamma(\tau) \in \bar{R}$.
\end{enumerate}

\begin{theorem}
Let $\Gamma$ be a congruence subgroup containing $(\begin{smallmatrix} 1 & 1 \\ 0 & 1 \end{smallmatrix})$. Then
\[ R_\Gamma \ := \ \bigcap_{\gamma \in \Gamma - \{ \pm I \}} \mathrm{outside}(\gamma) \]
is an open fundamental region for $\Gamma$ such that

(1) $|\mathrm{Re}(\tau)| \leq 1/2$ for all $\tau \in R_\Gamma$

(2) the imaginary part of $\tau \in \overline{R_\Gamma}$ is maximal in the orbit $\Gamma \tau$.
\end{theorem}

\begin{proof}
See \cite[Theorem 3.3]{Fer} and its proof.
\end{proof}

Thus the fundamental region $\mathfrak F_{\Gamma_0(p)}$ that we try to find lies between $R_{\Gamma_0(p)}$ and $\overline{R_{\Gamma_0(p)}}$.

\begin{lemma}
\[ R_{\Gamma_0(p)} \ = \ \{ \tau \in \mathbb H \, \big| \, |\mathrm{Re}(\tau)| < 1/2, \ 
|\tau - k/p| > 1/p \mbox{ for all } k = \pm 1, \pm 2, \ldots, \pm [p/2] \}. \]
\end{lemma}

\begin{proof}
Note that
\[
\bigcap_{\gamma \in \Gamma_0(p) - \{ \pm I \} \atop c_\gamma = 0} \mathrm{outside}(\gamma) \ = \ \{ \tau \in \mathbb H \, \big| \, |\mathrm{Re}(\tau)| < 1/2 \}
\]
and that
\[
\bigcap_{\gamma \in \Gamma_0(p) - \{ \pm I \} \atop |c_\gamma| = p} \mathrm{outside}(\gamma) \ = \ \{ \tau \in \mathbb H \, \big| \, |\tau - k/p| > 1/p \mbox{ for all } k \in \mathbb Z \mbox{ with } (k, p) = 1 \}
\]
because
\[
\gamma \in \Gamma_0(p), \ c_\gamma = \pm p \ \Longrightarrow \ \mathrm{outside}(\gamma) \ = \ \{ \tau \in \mathbb H \, \big| \, |\tau \mp a_\gamma / p| > 1/p \}.
\]
Thus we see that
\[ \bigcap_{\gamma \in \Gamma_0(p) - \{ \pm I \} \atop |c_\gamma| \leq p} \mathrm{outside}(\gamma) \ = \ \{ \tau \in \mathbb H \, \big| \, |\mathrm{Re}(\tau)| < 1/2, \ |\tau - k/p| > 1/p \mbox{ for all } k = \pm 1, \pm 2, \ldots, \pm [p/2] \}. \]
Let $\tau$ be a point contained in the set right above. If $|\mathrm{Re}(\tau)| \leq 1 / 2p$, then it is clear that $\tau \in \mathrm{outside}(\gamma)$ for any $\gamma \in \Gamma_0(p)$ with $|c_\gamma| \geq 2p$. If $|\mathrm{Re}(\tau)| > 1 / 2p$, then we see that $\mathrm{Im}(\tau) > \sqrt{3}/2p$ and so we deduce that $\tau \in \mathrm{outside}(\gamma)$ for any $\gamma \in \Gamma_0(p)$ with $|c_\gamma| \geq 2p$ because every point in $\mathrm{inside}(\gamma)$ has its imaginary part less than $1 / 2p$.
\end{proof}

%

Let $S_p = \{ \pm 1, \pm 2, \ldots, \pm \frac{p - 1}{2} \}$. For any $x \in \mathbb Z$ with $(x, p) = 1$, we denote by $x^{-1} \in S_p$ (respectively, $\langle x \rangle \in S_p$) the unique integer such that $xx^{-1} \equiv 1 \pmod p$ (respectively, $x \equiv \langle x \rangle \pmod p)$. For any $k \in S_p$, we define
\[
\gamma_k \ = \ \begin{pmatrix}
k & \frac{kk^{-1} - 1}{p} \\ p & k^{-1}
\end{pmatrix} \ \in \ \Gamma_0(p).
\]
Then we have shown in the above lemma that
\[ R_{\Gamma_0(p)} \ = \ \{ \tau \in \mathbb H \, \big| \, |\mathrm{Re}(\tau)| < 1/2 \} \bigcap \ \big(
\bigcap_{k \in S_p} \mathrm{outside}(\gamma_k) \big). \]

Notice that for any $k \in S_p$, the matrix $\gamma_k$ maps $\mathrm{arc}(\gamma_{-k^{-1}})$ onto $\mathrm{arc}(\gamma_k)$ taking endpoints to endpoints as follows:
\[ \gamma_{k}\Big(\frac{-k^{-1} \pm 1}{p}\Big) \ = \ \frac{k \mp 1}{p}. \]
Observe also that an elliptic point of order $2$ can occur only at the top points of the arcs and that an elliptic point of order $3$ can occur only at the points where the lines or arcs intersect, namely the points $\frac{k}{p} - \frac{1}{2p} + \frac{\sqrt{3}}{2p}i$ for $k \in S_p - \{ 1 \}$.

From these, we deduce that the top point of $\mathrm{arc}(\gamma_k)$ is an elliptic point of order $2$ if and only if $\mathrm{arc}(\gamma_{-k^{-1}})$ $=$ $\mathrm{arc}(\gamma_k)$ if and only if $k^2 \equiv -1 \pmod p$. So if we set
\[ E_p^{(2)} \ := \ \{ k \in S_p \, | \, k^2 \equiv -1 \pmod p \}, \]
then the points $\frac{k}{p} + \frac{1}{p}i$ with $k \in E_p^{(2)}$ are all the inequivalent elliptic points of order $2$ for $\Gamma_0(p)$. We further observe that for such a $k$, the matrix $\gamma_k$ maps $\mathrm{arc}(\gamma_k)$ onto itself with the direction reversed. Hence, if $\tau \in \partial R_{\Gamma_0(p)} \cap \mathrm{arc}(\gamma_k)$ for some $k \in E_p^{(2)}$, then we need to choose $\tau$ such that $\mathrm{Re}(\tau) \leq \frac{k}{p}$ in order to construct $\mathfrak F_{\Gamma_0(p)}$ satisfying the prescribed property (3). On the other hand, if $\tau \in \partial R_{\Gamma_0(p)} \cap (\mathrm{arc}(\gamma_{-k^{-1}}) \cup \mathrm{arc}(\gamma_k))$ for some $k \in S_p - E_p^{(2)}$, then we need to choose $\tau$ such that $\tau \in \mathrm{arc}(\gamma_{k_{(2)}})$, where
\[ k_{(2)} \ := \ \min \{ k, \, -k^{-1} \}. \]
For example, if $k = 1$, then $k_{(2)} = -1$ and so we need to choose points on $\mathrm{arc}(\gamma_{-1})$ instead of points on $\mathrm{arc}(\gamma_{1})$.

Now it remains to determine which points $\frac{k}{p} - \frac{1}{2p} + \frac{\sqrt{3}}{2p}i$ with $k \in S_p - \{ 1 \}$ are equivalent to each other under $\Gamma_0(p)$. For any $k \in S_p$ we see by a straightforward computation that
\[ \gamma_{k}\Big( \frac{1 - k^{-1}}{p} - \frac{1}{2p} + \frac{\sqrt{3}}{2p}i \Big) \ = \ \frac{k}{p} - \frac{1}{2p} + \frac{\sqrt{3}}{2p}i. \]
So the point $\frac{k}{p} - \frac{1}{2p} + \frac{\sqrt{3}}{2p}i$ is $\Gamma_0(p)$-equivalent to the point $\frac{1 - k^{-1}}{p} - \frac{1}{2p} + \frac{\sqrt{3}}{2p}i$ for any $k \in S_p$. Notice that the map $f: S_p - \{ 1 \} \longrightarrow S_p - \{ 1 \}$ taking $k$ to $\langle 1 - k^{-1} \rangle$ satisfies $f^3 = \mathrm{id}$, from which we infer that for any $k \in S_p - \{ 1 \}$ we have either $f(k) = k$ or ($k \neq f(k)$ and $f(k) \neq f^2(k)$ and $k \neq f^2(k)$). We see that
\begin{eqnarray*}
f(k) = k & \Longleftrightarrow & k + k^{-1} \equiv 1 \pmod p \\
& \Longleftrightarrow & k + k^{-1} = 1 \mbox{ or } 1 - p \\
& \Longleftrightarrow & k + k^{-1} = 1
\end{eqnarray*}
because $|k + k^{-1}| \leq p - 2$ by the assumption $p \geq 5$. Thus if $k^2 - k + 1 \equiv 0 \pmod p$, then $\gamma_k$ fixes the point $\frac{k}{p} - \frac{1}{2p} + \frac{\sqrt{3}}{2p}i$. On the other hand, if $k^2 - k + 1 \not\equiv 0 \pmod p$, then we see that the three distinct points $\frac{k}{p} - \frac{1}{2p} + \frac{\sqrt{3}}{2p}i$, $\frac{\langle 1 - k^{-1} \rangle}{p} - \frac{1}{2p} + \frac{\sqrt{3}}{2p}i$, $\frac{\langle 1 - (1 - k^{-1})^{-1} \rangle}{p} - \frac{1}{2p} + \frac{\sqrt{3}}{2p}i$ are equivalent to each other under $\Gamma_0(p)$. Therefore, if we set
\[ E_p^{(3)} \ := \ \{ k \in S_p \, | \, k^2 - k + 1 \equiv 0 \ (\bmod \ p) \}, \]
then the points $\frac{k}{p} - \frac{1}{2p} + \frac{\sqrt{3}}{2p}i$ with $k \in E_p^{(3)}$ are all the inequivalent elliptic points of order $3$ for $\Gamma_0(p)$.
Moreover, if $\tau$ is one of the three points above for some $k \in S_p - (\{ 1 \} \cup E_p^{(3)})$, then $\tau$ must be chosen as $\tau = \frac{k_{(3)}}{p} - \frac{1}{2p} + \frac{\sqrt{3}}{2p}i$, where
\[ k_{(3)} \ := \ \min \{ k, \, \langle 1 - k^{-1} \rangle, \, \langle 1 - (1 - k^{-1})^{-1} \rangle \}. \]
For example, if $k = -1$, then $k_{(3)} = -\frac{p-1}{2}$ and so the point $\tau = -\frac{1}{2} + \frac{\sqrt{3}}{2p}i$ must be chosen and the points $\tau = \pm\frac{3}{2p} + \frac{\sqrt{3}}{2p}i$ must be discarded.

Combining all of the above discussion, we conclude that $\mathfrak F_{\Gamma_0(p)}$ is the set of all points $\tau \in \mathbb H$ satisfying the following properties:
\begin{enumerate}
\item $|\mathrm{Re}(\tau)| \leq 1/2$

\item $|\tau - k/p| \geq 1/p$ for all $k \in S_p$

\item $|\mathrm{Re}(\tau)| = 1/2 \Longrightarrow \mathrm{Re}(\tau) = -1/2$

\item $|\tau - k/p| = 1/p$ with $k \in E_p^{(2)}$ $\Longrightarrow$ $\mathrm{Re}(\tau) \leq k/p$

\item $|\tau \pm 1/p| = 1/p \Longrightarrow \mathrm{Re}(\tau) \leq 0$

\item $|\tau - k/p| = 1/p$ with $k \in S_p - (\{ \pm 1 \} \cup E_p^{(2)})$ $\Longrightarrow$ $\mathrm{Re}(\tau) \leq (2k_{(2)} + 1)/ 2p$

\item $\tau \neq (2k - 1) / 2p + i\sqrt{3}/2p$ for all $k \in S_p - (\{ 1 \} \cup E_p^{(3)})$ with $k \neq k_{(3)}$.
\end{enumerate}

\begin{remark}
Using the same method as above, we can obtain fundamental regions for $\Gamma_0(N)$ with $N= 2, 3, 4$ as follows:
\begin{eqnarray*}
\mathfrak F_{\Gamma_0(N)} & = & \big\{ \tau \in \mathbb H \, \big| \, |\mathrm{Re}(\tau)| \leq 1/2, \ |\tau - 1/N| \geq 1/N, \ |\tau + 1/N| \geq 1/N, \\
& & \qquad \quad \ \ \big( |\mathrm{Re}(\tau)| = 1/2 \mbox{ or } |\tau - 1/N| = 1/N \mbox{ or } |\tau + 1/N| = 1/N \big) \Longrightarrow \mathrm{Re}(\tau) \leq 0 \big\}.
\end{eqnarray*}
\end{remark}

\vspace{.3cm}

\textit{Acknowledgments.} The author would like to express his sincere thanks to the anonymous referee for his or her careful reading on the manuscript.

%

\bibliographystyle{amsplain}

\begin{thebibliography}{10}

\bibitem{Cho2} B. Cho, \textit{Modular equations for congruence subgroups of genus zero}, Ramanujan J., 51 (2020), 187--204.

\bibitem{Cho3} B. Cho, \textit{Modular equations for congruence subgroups of genus zero (II)}, submitted for publication.

\bibitem{CKP1} B. Cho, J. K. Koo, and Y. K. Park, \textit{Arithmetic of the Ramanujan--G{\"o}llnitz--Gordon continued fraction}, J. Number Theory, 129 (2009), 922--948.

\bibitem{CK} S. Choi and C. H. Kim, \textit{Recursion formulas for modular traces of weak Maass forms of weight zero}, Bull. Lond. Math. Soc., 42 (2010), 639--651.

\bibitem{Coh} H. Cohen, \textit{A Course in Computational Algebraic Number Theory}, Graduate Texts in Mathematics (Volume 138), Springer, 2000.

\bibitem{Cox} D. A. Cox, \textit{Primes of the form $x^2 + ny^2$}, second edition, Wiley, 2013.

\bibitem{Fer} C. R. Ferenbaugh, \textit{The genus-zero problem for $n|h$-type groups}, Duke Math. J., 72 (1993), 31--63.

\bibitem{Lang} S. Lang, \textit{Elliptic functions}, second edition, Graduate Texts in Mathematics 112, Springer-Verlag, New York, 1987.

\bibitem{Mura} Y. Murakami, \textit{Intersection numbers of modular correspondences for genus zero modular curves}, J. Number Theory, 209 (2020), 167--194.

\bibitem{Shi} G. Shimura, \textit{Introduction to the arithmetic theory of automorphic functions}, in: Kan\^{o} Memorial Lectures, No. 1, in: Publications of the Mathematical Society of Japan, vol. 11, Iwanami Shoten Publishers/Princeton University Pres, Tokyo/Princeton, N.J., 1971, xiv+267 pp.

\bibitem{Zagier} D. Zagier, \textit{Traces of singular moduli}, Motives, polylogarithms and Hodge theory, Part I, International Press Lecture Series 3 (International Press, Somerville, MA, 2002), 211--244.

\end{thebibliography}

\end{document}